\documentclass[final,3p,mathptmx]{article}

\usepackage{graphicx}
\usepackage{amsmath}
\usepackage{amsfonts}
\usepackage{amssymb}
\usepackage{amsthm}
\usepackage{newlfont}
\usepackage{longtable}
\usepackage{verbatim}

\usepackage{color}
\usepackage{tikz}
\usepackage{mathrsfs}

\newcommand{\R}{{\mathbb R}}
\newcommand{\N}{{\mathbb N}}

\newcommand{\RR}{{\mathbb R}}

\newcommand{\Rd}{{\R}^d}
\newcommand{\Comp}{\mathrm{K}(\Rd)}
\newcommand{\Conv}{\mathrm{Co}(\Rd)}

\newcommand{\inK}{\in\Comp}

\newcommand{\co}{\mathrm{co}}

\newcommand{\haus}{\mathrm{haus}}
\newcommand{\dist}{\mathrm{dist}}

\newcommand{\CH}{{\mathrm{CH}}}
\newcommand{\Graph}{{\mathrm{Graph}}}

\newcommand{\PrjXonY}[2]{\Pi_{#2}{(#1)} }
\newcommand{\Pair}[2]{\Pi \big ( {#1},{#2} \big )}
\newcommand{\MetInt}{{\scriptstyle(\cal M)}\int_{a}^{b}F(x)dx}
\newcommand{\WeightedMetInt}{{\scriptstyle(\cal M_{k})}\int_{a}^{b}k(x)F(x)dx}
\newcommand{\MetIntSpec}[2]{{\scriptstyle(\cal M)}\int_{#1}^{#2}}
\newcommand{\IntSel}[1]{\int_{a}^{b} {#1} (x)dx}

\newcommand{\Map}[1]{F:#1 \rightarrow \Comp }

\newcommand{\CBV}{\mathrm{CBV}[a,b]}
\newcommand{\BV}{\mathrm{BV}[a,b]}

\newcommand{\calF}{\cal{F}[a,b]}

\newcommand{\ModulCont}[2]{\omega\big( {#1},{#2} \big)}
\newcommand{\LeftLocalModul}[3]{\omega^{-}\big( {#1},{#2},{#3} \big)}
\newcommand{\RightLocalModul}[3]{\omega^{+}\big( {#1},{#2},{#3} \big)}
\newcommand{\NewLeftLocalModul}[3]{\varpi^{-}\big( {#1},{#2},{#3} \big)}
\newcommand{\NewRightLocalModul}[3]{\varpi^{+}\big( {#1},{#2},{#3} \big)}
\newcommand{\LocalModulCont}[3]{\omega\big( {#1},{#2},{#3} \big)}
\newcommand{\EqvClass}[4]{\mathscr{BV}_{#1}\big( {#2},{#3},{#4} \big)}

\newcommand{\eps}{\varepsilon}

\newcommand{\SymbText}[1]{\hbox to 10cm { \dotfill #1 \dotfill}}

\newtheorem{remark}{Remark}[section]

\newtheorem{theorem}[remark]{Theorem}
\newtheorem{propos}[remark]{Proposition}
\newtheorem{corol}[remark]{Corollary}
\newtheorem{lemma}[remark]{Lemma}
\newtheorem{result}[remark]{Result}
\newtheorem{example}[remark]{Example}
\newtheorem{defin}[remark]{Definition}
\newtheorem{pict}[remark]{Figure}

\begin{document}


\title {Metric Fourier approximation of set-valued functions\\ of bounded variation}

\renewcommand{\thefootnote}{\fnsymbol{footnote}}
\author{
Elena E. Berdysheva \footnotemark[1], \quad
Nira Dyn \footnotemark[2], \quad
Elza Farkhi \footnotemark[2]\; \footnotemark[3],\quad
Alona Mokhov \footnotemark[4]\;
}
\footnotetext[1]{Justus Liebig University Giessen, Germany }
\footnotetext[2]{Tel-Aviv University, School of Mathematical Sciences}
\footnotetext[3]{On leave from Institute of Mathematics and Informatics, Bulgarian Academy of Sciences, Sofia}
\footnotetext[4]{Afeka, Tel-Aviv Academic College of Engineering}
\date{}
\maketitle

\medskip

{ \small {\normalsize \textbf {Abstract.}}
 We introduce and investigate an adaptation of Fourier series to set-valued functions (multifunctions, SVFs) of bounded variation.  In our approach we define an analogue of the partial sums of the Fourier series  with the help of the Dirichlet kernel using the newly defined weighted metric integral. We derive error bounds for these approximants. As a consequence, we prove that the sequence of the partial sums converges pointwisely in the Hausdorff metric to the values of  the approximated set-valued function at its points of continuity, or to a certain set described in terms of the metric selections of the approximated multifunction at a point of discontinuity. Our error bounds are obtained with the help of the new notions of one-sided local moduli and quasi-moduli of continuity which we discuss more generally for functions with values in metric spaces.

}


\noindent{ \small {\normalsize \textbf{Mathematics Subject Classification 2020:}} 26E25, 28B20, 28C20, 54C60, 54C65, 42A20, 42A99


\section {Introduction} \label{Sect_Intro}

Set-valued functions (SVFs, multifunctions) find applications in different fields such as economy, optimization, dynamical systems, control theory, game theory, differential inclusions, geometric modeling. 
Analysis of set-valued functions  has been a rapidly developing field in the last decades. One may consider the book \cite{AubinFrankowska:90} as establishing the field of set-valued analysis.  Approximation of SVFs has been developing in parallel.

Older approaches to the approximation, related mainly to control theory, investigate almost exclusively  SVFs with convex images (values). 
Research on approximation and numerical integration of set-valued functions with convex images can be found e.g. in \cite{VIT:79, ADontchevFarkhi:87, Veliov:int89, MNikolskii:89dokl, MNikolskii:Opt90, MNikolskii:deriv90, MNikolskii:91a, TDonchevFarkhi:90, BaierLempio:94a,BaierLempio:94b, Lempio:95, Baier:95thesis, DynFarkhi:00, DF:04, BaierFarkhi:07, Muresan:SVApprox2010, BaierPerria:11, Babenko:2016, Campiti:2019}. 
The standard tools used  are the Minkowski linear combinations and the Aumann integral.  It is well-known that the Aumann integral of a multifunction with compact values in $\Rd$  is convex even if the values of the integrand are not convex \cite{AUM:65}.  This property is called convexification, see e.g.~\cite{DF:04}. Also the Minkowski convex combinations with a growing number of summands suffer from convexification~\cite{DF:04}. 

Some newer applications, as geometric modeling for instance, motivate the study of  approximation of SVFs with general, not necessarily convex values. 
Trying to apply the known methods for the convex-valued case to set-valued functions with general values,
R.~A.~Vitale cosidered in \cite{VIT:79} the polynomial Bersntein operators 
adapted to SVFs by replacing linear combinations of numbers by 
the Minkowski linear combinations of sets.  While this construction works perfectly for SVFs with convex images, in the general case the sequence of so generated Bernstein approximants does not approximate the given SVF but 
the multifunction with values equal to the convex hulls of those of the original SVF. Clearly, such methods are useless for approximating set-valued functions with general, not necessarily convex images.

A pioneering work on approximation of SVFs with general images is done by Z.~Artstein~\cite{Artstein:MA}, who  constructs piecewise-linear interpolants of multifunctions. He replaces the Minkowski averages between two sets by the set of averages of special 
pairs of elements termed in later works ``metric pairs''.  Using the concept of metric pairs and metric linear combinations, N.~Dyn, E.~Farkhi and A.~Mokhov developed in  a series of works
techniques that are free of convexification and are suitable for approximating set-valued functions with general compact images. 
The tools used in these techniques include {repeated binary metric averages \cite{DynFarkhi:01, DynMokhov:06,DFM:Book_SV-Approx},}
metric linear combinations \cite{DFM:Chains,DFM:Book_SV-Approx}, metric selections \cite{DFM:Book_SV-Approx, DFM:MetricIntegral}
and the metric integral \cite{DFM:MetricIntegral}, which is extended here to a weighted metric integral. In~\cite{BDFM:2019, DFM:Chains, DFM:Book_SV-Approx, DFM:MetricIntegral} the authors studied approximation of set-valued functions by means of metric adaptations of classical approximation operators such as the Bernstein polynomial operator, the Schoenberg spline operator, the polynomial interpolation operator.
While in older papers the approximated SVFs are mainly continuous, the later works \cite{DFM:MetricIntegral,BDFM:2019} 
are concerned with multifunctions of bounded variation.  

The main topic  is an adaptation of the trigonometric Fourier series to set-valued functions of bounded variation with general compact images. We also try to obtain error bounds under minimal regularity requirements on the multifunctions to be approximated and focus  on the investigation on SVFs of bounded variation. We use in our analysis some properties of maps of bounded
variation with values in metric spaces proved in \cite{Chistyakov:On_BV-mappings}.

We are familiar only with few works on trigonometric approximation of multifunctions.  Some results on this topic for  convex-valued SVFs by methods based on the Aumann integral are obtained in \cite{Babenko:2016}. For the related topic of trigonometric approximation of fuzzy-valued functions see, e.g.~\cite{Anastassiou:Book2010, Bede:2005, Yavuz:2019, KadakBasar:2014}. Note that in this context the  level sets determine multifunctions with convex values (intervals in $\R$).

In this paper we define the metric analogue of  the partial sums of the Fourier series of a multifunction via convolutions  with the Dirichlet kernel of order $n$, for $n\ge 0$, the convolutions being defined as weighted metric integrals.
To study error bounds of these approximants and to prove convergence as $n\to \infty$, we introduce new one-sided local moduli of continuity in Section~\ref{Sect_Prelim_Regularity} and quasi-moduli of continuity in Section~\ref{Sect_Fourier}. 
The main result of the paper is analogous to the classical Dirichlet-Jordan Theorem for real functions~\cite{Zygmund:TrigSeries}. It states
the pointwise convergence in the Hausdorff metric of the metric Fourier approximants of a multifunction of bounded variation
to a compact set. 
In  particular, if the multifunction $F$ is of bounded variation and continuous at a point $x$, then the metric Fourier approximants of it at $x$ converge to $F(x)$. The convergence is uniform 
in closed finite intervals where $F$ is continuous. At a point of discontinuity the limit set is determined by the values of the metric selections of $F$ there.

The paper is organized as follows. In the next section some basic notions and notation are recalled.
One-sided local moduli of continuity of univariate functions with values in a metric space are introduced  and studied in Section 3. The theory developed in Section 3 is specified in Section 4 to set-valued functions of bounded variation, to their
chain functions and metric selections. In Section 5 the weighted metric integral is introduced and some of its properties are derived. The main results of the paper are presented in Section 6. To make the reading easier, the section is divided into three subsections. The first subsection contains the definition of the metric Fourier approximants of multifunctions. The second subsection contains a refinement of the classical Dirichlet-Jordan Theorem~\cite{Zygmund:TrigSeries}. There we obtain error bounds for the Fourier approximants 
for special classes of real functions of bounded variation. This refinement is used in the third subsection for the main results on the metric Fourier approximation of set-valued functions. In Section 7 we discuss properties of a set-valued function and of its metric selections at a point of discontinuity and  study the structure of the limit set of the metric Fourier approximants.

There are two appendices: Appendix A contains the proof of Theorem~\ref{theorem_PW-limit_of_MS_isMS} which is 
stated without a proof in Section 4 of \cite{DFM:MetricIntegral}.  Appendix B contains the proof of the refined Dirichlet-Jordan Theorem from Subsection~\ref{Subsec_FourierClasses}.


\section {Preliminaries}\label{Sect_Prelim}
In this section we introduce some notation and basic notions related to sets and set-valued functions. 

All sets considered from now on are sets in $\Rd$.  We denote by $\Comp$\label{CompSets} the collection of all compact non-empty subsets of~$\Rd$. By $\Conv$\label{ConvSets} we denote the collection of all convex sets in $\Comp$. The convex hull of a set $A$ is denoted by $\co(A)$\label{convex_hull}. 
The metric in $\Rd$ is of the form $\rho(u,v)=|u-v|$, where $|\cdot|$ is a norm on $\Rd$. 
Note that all norms on $\Rd$ are equivalent.  In the following we fix one norm in $\Rd$.  Recall that $\Rd$ is a complete metric space.

Let $A$ and $B$ be non-empty subsets of~$\Rd$. To measure the distance between $A$ and $B$, we use the Hausdorff metric based on $\rho$
\begin{equation}\label{defin_Haus}
	\haus(A,B)_{\rho}= \max \left\{ \sup_{a \in A}\dist(a,B)_{\rho},\; \sup_{b \in B}\dist(b,A)_{\rho} \right\},
\end{equation}
where the distance from a point $c$ to a set $D$ is $\dist(c,D)_{\rho}=\inf_{d \in D}\rho(c,d)$.

It is well known that $\Comp$ and $\Conv$ are complete metric spaces with respect to the Hausdorff metric~\cite{RockWets, {S:93}}.
For an arbitrary metric space $(X,\rho)$, the same formula \eqref{defin_Haus} defines a metric on the set $\mathcal{C}(X)$ of all non-empty closed subsets of $X$. It is known that the metric space $(\mathcal{C}(X), \haus)$ is complete if $(X,\rho)$ is complete. Moreover, $(\mathcal{C}(X), \haus)$ is compact if $X$ is compact (e.g. ~\cite[Section 4.4]{AmbrosioTilli:Topics}).

We denote by $|A|=\haus (A,\{0\})$ the ``norm'' of the set $A \in \Comp$.

The set of projections of $a \in \Rd$ on a set $B \inK$ is
$$
    \PrjXonY{a}{B}=\{b \in B \ : \ |a-b|=\dist(a,B)\},
$$
and the set of metric pairs of two sets $A,B \inK$ is
$$
\Pair{A}{B} = \{(a,b) \in A \times B \ : \ a \in \PrjXonY{b}{A}\;\, \mbox{or}\;\, b\in\PrjXonY{a}{B} \}.
$$
Using metric pairs, we can rewrite
$$
    \haus(A,B)= \max \{|a-b| \ :\ (a,b)\in \Pair{A}{B}\}.
$$

In \cite{DFM:MetricIntegral}, the three last-named authors introduced the notions of  a metric chain and of a metric linear combination as follows.

\begin{defin}\label{Def_MetChain_MetSelection} \cite{DFM:MetricIntegral}
  Given a finite sequence of sets $A_0, \ldots, A_n \in \Comp$, $n \ge 1$,  a metric chain of $A_0, \ldots, A_n$ is an $(n+1)$-tuple $(a_0,\ldots,a_n)$ such that $(a_i,a_{i+1}) \in \Pair {A_i}{A_{i+1}}$, $i=0,1,\ldots,n-1$. We denote the collection of all metric chains of $A_0, \ldots, A_n$ by 
$$
	\CH(A_0,\ldots,A_n)= \left\{ (a_0,\ldots,a_n) \ : \ (a_i,a_{i+1}) \in \Pair {A_i}{A_{i+1}}, \  i=0,1,\ldots,n-1 \right\}.
$$
 The metric linear combination of the sets $A_0, \ldots, A_n \in \Comp$, $n \ge 1$, is
$$
    \bigoplus_{i=0}^n \lambda_i A_i =
    \left\{ \sum_{i=0}^n \lambda_i a_i \ : \ (a_0,\ldots,a_n) \in \CH(A_0,\ldots,A_n) \right\}, \quad \lambda_0,\ldots,\lambda_n \in \R.
$$
\end{defin}

Note that the metric linear combination depends on the order of the sets, in contrast to the Minkowski linear combination of sets which is defined by
$$
    \sum_{i=0}^n \lambda_i A_i =
    \left\{ \sum_{i=0}^n \lambda_i a_i \ : \  a_i \in A_i \right\},\quad n \ge 1.
$$

For a sequence of sets $\{A_n\}_{n=1}^{\infty}$ the lower Kuratowski limit is the set of all limit points of converging sequences
$\{a_{n}\}_{n=1}^{\infty}$, where $a_{n} \in A_{n} $,  namely,
$$
    \liminf_{n \to \infty} A_n = \left\{a \ : \  \exists \, a_{n} \in A_{n} \text{ such that } \lim_{n \to \infty}a_{n} = a \right\}.
$$
Analogously, for a set-valued function $F:[a,b]\to \Comp$ and $\widetilde{x} \in [a,b]$  we define
$$
\liminf_{x \to \widetilde{x}} F(x) = \left\{y \ : \  \forall \, \{x_k\}_{k=1}^{\infty} \subset [a,b] \ \text{with} \  x_k\to\widetilde{x} \ \exists \, \{y_k\}_{k=1}^{\infty} \ \text{with} \ y_k\in F(x_k),  k\in\N, \ \text{and} \ y_k \to y
 \right\}.
$$

The upper Kuratowski limit is the set of all limit points of converging subsequences
$\{a_{n_k}\}_{k=1}^{\infty}$, where ${a_{n_k} \in A_{n_k} }$, $k\in \N$, namely
$$
    \limsup_{n \to \infty} A_n = \left\{a \ : \  \exists\, \{n_k\}_{k=1}^{\infty},\, n_{k+1}>n_k,\, k\in \N,\ \exists \, a_{n_k} \in A_{n_k} \text{ such that } \lim_{k \to \infty}a_{n_k} = a \right\}.
$$
Correspondingly, for a set-valued function $F:[a,b]\to \Comp$ and $\widetilde{x} \in [a,b]$
$$
\limsup_{x \to \widetilde{x}} F(x) = \left\{y \ : \  \exists \, \{x_k\}_{k=1}^{\infty} \subset [a,b] \ \text{with} \  x_k\to\widetilde{x} \ \exists \, \{y_k\}_{k=1}^{\infty} \ \text{with} \ y_k\in F(x_k),  k\in\N, \ \text{and} \ y_k \to y
 \right\}.
$$

A sequence $\{A_n\}_{n=1}^{\infty}$ converges in the sense of Kuratowski to $A$ if ${\displaystyle A =  \liminf_{n \to \infty} A_n  = \limsup_{n \to \infty} A_n }$. 
Similarly, a set $A$ is a Kuratowski limit of $F(x)$ as $x \to \widetilde{x}$ if  ${\displaystyle A =  \liminf_{x \to \widetilde{x}} F(x)  = \limsup_{x \to \widetilde{x}} F(x) }$. 

\begin{remark}\label{Remark_Kurat=Haus}
There is a connection between convergence in the  sense of Kuratowrski and convergence in the Hausdorff metric, the latter meaning that $\displaystyle \lim_{n \to \infty} { \haus (A_n,A)} = 0$ or $\displaystyle \lim_{x \to \widetilde{x}} { \haus \big( F(x), A \big)} = 0$,  respectively.  If the underlying space $X$ is compact, then convergence in the Hausdorff metric and in the sense of Kuratowski are equivalent (see, e.g.,~\cite[Section 4.4]{AmbrosioTilli:Topics}).
\end{remark}


\section {Regularity measures of functions with values in a metric space}\label{Sect_Prelim_Regularity}

Here we consider regularity measures of functions defined on a fixed interval $[a,b] \subset \R$ with values in a complete metric space $(X,\rho)$.

A basic notion in this paper is the \textbf{modulus-bounding function} $\omega(\delta)$ which is a non-decreasing function $\omega: [0,\infty) \rightarrow [0,\infty)$. Frequently we occur the situation when in addition $\lim\limits_{\delta \to 0^+} \omega(\delta)=0$, but we do not require this in the definition.

In the analysis of continuity of a function at a point, the notion of the local modulus of continuity is instrumental~\cite{SendovPopov}
\begin{equation}\label{Def_ClassicalModuliContin}
\LocalModulCont{f}{x^*}{\delta} = \sup \left\{\, \rho(f(x_1),f(x_2)): \quad x_1,x_2 \in \left[x^*-\frac{\delta}{2},x^*+\frac{\delta}{2} \right]\cap[a,b] \,\right\},\quad \delta >0.
\end{equation}

To characterize left and right continuity of functions, we introduce the left and the right  local moduli of continuity, respectively.

\begin{defin}
The left local modulus of continuity of $f$ at $x^*\in[a,b]$ is
\begin{equation}\label{defin_LeftModul}
\LeftLocalModul{f}{x^*}{\delta}=\sup \left\{ \rho(f(x),f(x^*)) \ :\ x\in [x^*-\delta, x^*] \cap [a,b] \right \},\quad \delta>0.
\end{equation}
Similarly, the right local modulus of continuity of $f$ at $x^* \in [a,b]$ is
\begin{equation}\label{defin_RightModul}
\omega^{+}(f,x^*,\delta)=\sup \left \{\rho(f(x),f(x^*)) \ : \ x\in [x^*,x^*+\delta] \cap [a,b] \right \},\quad \delta>0.
\end{equation}
\end{defin}

\begin{remark}\label{Remark_Moduli}
	\begin{enumerate}
		\item[]${}$
		\item[(i)] One can define the one-sided  local moduli of continuity analogously to~\eqref{Def_ClassicalModuliContin}, for example, the left local modulus as 
		$$
		\nu^{-}(f,x^*,\delta)=\sup \left \{ \rho(f(x_1),f(x_2))\ : \ x_1,x_2\in [x^*-\delta, x^*] \cap [a,b] \right \},\quad \delta>0.
		$$
		Yet it is easily seen that this quantity is equivalent to~\eqref{defin_LeftModul}, namely
		$$
		\omega^{-}(f,x^*,\delta) \le \nu^{-}(f,x^*,\delta) \le 2\omega^{-}(f,x^*,\delta).
		$$
		\item[(ii)] Note that the classical global modulus of continuity $\displaystyle \ModulCont{f}{\delta} = \sup_{x\in[a,b]} \LocalModulCont{f}{x}{\delta}$ is subadditive in $\delta$, while this property is not satisfied by the local moduli.
	\end{enumerate}
\end{remark}
\medskip

The following relations hold for $x^* \in [a,b]$:
\begin{equation}\label{LocalModules-Prop1}
\max \{ \omega^{-}(f,x^*,\delta), \omega^{+}(f,x^*,\delta) \} \le \LocalModulCont{f}{x^*}{2\delta} ,
\end{equation}
$$
\LocalModulCont{f}{x^*}{\delta} \le 2\max \left \{ \omega^{-}\left(f,x^*,\delta/2 \right),  \omega^{+}\left(f,x^*,\delta/2\right) \right \}, \quad \delta > 0.
$$

In the next proposition we extend some properties known for the local modulus of continuity $\omega(f,x^*,\delta)$ to the one-sided local moduli.

\begin{propos}\label{Prop_LocalModul_Vanishes}
A function  $f:[a,b]\to X$ is left continuous at $x^* \in (a,b]$ if and only if $\lim\limits_{\delta \to 0+}\LeftLocalModul{f}{x^*}{\delta}= 0$.  
The function  $f$ is right continuous at $x^* \in [a,b)$ if and only if $\lim\limits_{\delta \to 0+}\omega^{+}(f,x^*,\delta)=0$.
\end{propos}

\begin{proof}
A function $f$ is left continuous at $x^*$ if and only if for every $\varepsilon >0$ there exists $\delta>0$ such that $\rho(f(x),f(x^*))<\varepsilon$ for all $x\in [x^*-\delta, x^*]\cap[a,b]$. This implies that
$$
\LeftLocalModul{f}{x^*}{\delta}=\sup \left\{ \rho(f(x),f(x^*)) \ : \ x\in [x^*-\delta, x^*]\cap[a,b] \right\} < \varepsilon.
$$
Since, by definition, $\LeftLocalModul{f}{x^*}{\delta}$ is non-increasing in $\delta$, the above is equivalent to  $\lim\limits_{\delta \rightarrow 0+}\LeftLocalModul{f}{x^*}{\delta}=0$.
The proof for $\omega^{+}(f,x^*,\delta)$ is similar.
\end{proof}

We recall the notion of the variation of a function  ${f:[a,b]\rightarrow X}$. Let $\chi=\{x_0,\ldots, x_n\} $, $a=x_0 < \cdots <x_n=b$, be a partition of the interval $[a,b]$ with the norm 
$$
|\chi|=\max_{0\le i\le n-1} (x_{i+1}-x_i).
$$ 
The variation of $f$ on the partition $\chi$ is defined as
$$
V(f,\chi) = \sum_{i=1}^{n} \rho(f(x_i),f(x_{i-1})).
$$
The total variation of $f$ on $[a,b]$ is 
$$
V_{a}^{b}(f) = \sup_{\chi} V(f,\chi),
$$
where the supremum is taken over all partitions $\chi$ of $[a,b]$.

A function $f$ is said to be of bounded variation if ${ V_{a}^{b}(f) < \infty}$. We call functions of bounded variation BV functions and write $f \in \BV$. If $f$ is also continuous,  we write $f\in \CBV$.

For $f \in \BV$ the  function $v_f:[a,b]\rightarrow \R$,\, $v_f(x)=V_{a}^{x}(f)$ is called the variation function of $f$. Note that 
$$
	V_{z}^{x}(f)=v_f(x)-v_f(z) \quad  \mbox{for} \quad a\le z<x \le b,
$$
and that $v_f$ is monotone non-decreasing.

\begin{propos}\label{Prop_LocalModul_f<LocalModul_v_f}
For a function ${f:[a,b]\to X}$, ${f \in \BV}$ we have
$$
\LeftLocalModul{f}{x^*}{\delta}\le\LeftLocalModul{v_f}{x^*}{\delta} \quad \text{and} \quad \omega^{+}(f,x^*,\delta) \le \omega^{+}(v_f,x^*,\delta),
\quad x^* \in [a,b], \quad \delta > 0.
$$
\end{propos}
\begin{proof}
We prove only the first inequality, the proof of the second one is similar. 

\noindent If $x^* = a$ then both sides of the inequality are zero and the claim follows. For $x^* \in (a,b]$  we have
\begin{align*}
&\LeftLocalModul{f}{x^*}{\delta} =\sup \{\rho(f(x),f(x^*)) \ : \  \max\{x^*-\delta, a\}  \le x\le x^* \} \le \sup \{ V_x^{x^*}(f) \ :\  \max\{x^*-\delta, a\}  \le x\le x^* \} \\
& = \sup \{v_f(x^*)-v_f(x) \ : \   \max\{x^*-\delta, a\}  \le x\le x^* \} 
= \sup \{ |v_f(x^*)-v_f(x)| \ :\  \max\{x^*-\delta, a\}  \le x\le x^* \} \\
& =\LeftLocalModul{v_f}{x^*}{\delta}.
\end{align*}
\end{proof}
The following claim is a slight refinement of Proposition~1.1.1 in~\cite{DFM:Book_SV-Approx} and of~\cite[Chapter~9, Sec.~32, Theorem~3]{Kolmogorov:HellyTheorems}. 

\begin{propos}\label{Prop_LeftContEquiv_v_f-f}
A function $f:[a,b]\to X$, $f \in \BV$ is left continuous at $x^*\in (a,b]$ if and only if $v_f$ is left continuous at~$x^*$.
The function $f$ is right continuous at $x^* \in [a,b)$ if and only if $v_f$ is right continuous at~$x^*$.
\end{propos}
\begin{proof} We prove only the first statement, the proof of the second one is similar. 

If $v_f$ is left continuous at $x^*$,  then by Propositions~\ref{Prop_LocalModul_f<LocalModul_v_f} and~\ref{Prop_LocalModul_Vanishes} also $f$ is left continuous at $x^*$. Now we prove the other direction. We closely follow the proof in \cite[Chapter 9, Sec. 32, Theorem 3]{Kolmogorov:HellyTheorems}.

Assume that $f$ is left continuous at $x^*$.  Then for each $\varepsilon>0$ there exists $\delta>0$ such that
\begin{equation}\label{ref_contin}
	\rho(f(x),f(x^*))<\varepsilon /2, \quad x\in (x^*-\delta, x^*).
\end{equation}
The definition of the total variation implies that one can choose a partition $\chi=\{a=x_0<x_1<\cdots<x_n=x^*\}$ such that
$$
V_{a}^{x^*}(f) < V(f,\chi) + \varepsilon/2 = \sum_{i=1}^{n} \rho(f(x_i),f(x_{i-1})) + \varepsilon/2.
$$
Adding more points to $\chi$ if necessary, we can guarantee that $0 < x^*-x_{n-1} < \delta$. Then by~\eqref{ref_contin} we have $\rho(f(x^*),f(x_{n-1}))<\varepsilon /2$.
Thus,
$$
v_f(x^*)=V_{a}^{x^*}(f) < \sum_{i=1}^{n-1} \rho(f(x_i),f(x_{i-1})) + \varepsilon \le v_f(x_{n-1})+\varepsilon,
$$
and consequently $v_f(x^*) - v_f(x_{n-1}) < \varepsilon$. Put $\delta^*= x^*-x_{n-1} > 0$. By the monotonicity of $v_f$,
$$
v_f(x^*) - v_f(x) < \varepsilon
$$
holds for all $x\in (x^* - \delta^*,x^*)$. This means that $v_f$ is left continuous at $x^*$.
\end{proof}

Analogs of Propositions~\ref{Prop_LocalModul_f<LocalModul_v_f} and \ref{Prop_LeftContEquiv_v_f-f}  for the two-sided local modulus of continuity are well-known: 
\begin{propos}\label{Corol_ContEquiv_v_f-f}
For a function ${f:[a,b]\to X}$, ${f \in \BV}$ we have
$$
\omega(f,x^*,\delta) \le \omega(v_f,x^*,\delta), \quad x^* \in [a,b], \quad \delta > 0.
$$
Moreover, $f$  is continuous at $x^* \in [a,b]$ if and only if $v_f$ is continuous at~$x^*$.
\end{propos}
The first statement can be proved along the same lines, and the second statement follows immediately from Proposition~\ref{Prop_LeftContEquiv_v_f-f}. 

\begin{remark}
	Note that, in general, $\omega(f,x^*,\delta)$ and $\omega(v_f, x^*,\delta)$ are not equivalent for $f \in \BV$. As an example, consider $f(x) = x^2 \sin{\frac{1}{x}} \in \mathrm{BV}[0,1]$ (where we define $f(0) = 0$ by continuity). It is easy to see that 
	$$
	\omega(f,0,\delta) = \sup{ \left\{ |f(x_1) - f(x_2)| : x_1,x_2 \in \left[ 0, \delta/2 \right] \right\} }
	\le 2 \left( \frac{\delta}{2} \right)^2 = \frac{\delta^2}{2}, \quad \delta > 0.
	$$
	To estimate the local variation of $f$, consider the points $\frac{1}{x_k} = \frac{\pi}{2} + \pi k$, $k \in \N$, so that $\sin{ \frac{1}{x_k} } = (-1)^k$. Then
	$$
	\omega(v_f, 0, \delta) 
	=  V_0^{\delta/2}(f) \ge 2 \sum_{ k > \frac{2}{\delta \pi} - \frac{1}{2}} \left( \frac{1}{\frac{\pi}{2} + \pi k} \right)^2
	\ge \frac{2}{\pi^2} \sum_{ k > \frac{2}{\delta \pi} + \frac{1}{2}} \frac{1}{k^2} \sim \delta.
	$$
\end{remark}

Helly's Selection Principle (see, e.g.~\cite[Chapter 6]{Kolmogorov:HellyTheorems}) will be heavily used  in our analysis. We cite a version of it which is relevant to our paper.  

\noindent \textbf{Helly's Selection Principle}. \textit{ Let $\{f_n\}_{n \in \N}$ be a sequence of functions $f_n : [a,b] \to \RR$, and assume that there are constants $A,B > 0$ such that  $|f_n(x)| \le A$, $n \in \N$, $x \in [a,b]$ and $V_a^b(f_n) \le B$, $n \in \N$.  Then $\{f_n\}_{n \in \N}$ contains a subsequence $\{f_{n_k}\}_{k \in \N}$ that converges pointwisely to a function $f^{\infty} : [a,b] \to \RR$, i.e., $f^{\infty}(x) = \lim_{k \to \infty} f_{n_k}(x)$, $x \in [a,b]$.
}
\medskip

In the following statements we consider pointwise limits of sequences of BV functions. We show that the limit function inherits local properties which are shared by the members of the sequence. The first result is known and is given here for the readers' convenience.

\begin{theorem}\label{theorem_LimitFunc_BV}
	Let $\{f_n\}_{n=1}^\infty$ be a sequence of functions $f_n : [a,b]\rightarrow X$ that converges pointwisely to a function $f^{\infty} : [a,b]\rightarrow X$. Then
	$$
	V_{a}^{b}(f^\infty) \le \liminf_{n \to \infty} V_a^b(f_n).
	$$
In particular, if $V_{a}^{b}(f_n) \le A$ for all $n\in \N$ with some $A \in \RR$, then
	$$
	V_{a}^{b}(f^\infty) \le A.
	$$
\end{theorem}
\begin{proof}
	Let $\varepsilon>0$ be arbitrarily small and let $\chi=\{ x_0, \ldots ,x_{K}\}_{}$, $a = x_0 < \cdots <x_{K}=b$, be an arbitrary partition of $[a,b]$. There exists a subsequence $\{f_{n_k}\}_{k \in \N}$ satisfying $\displaystyle \lim_{k \to \infty} V_a^b(f_{n_k}) = \liminf_{n \to \infty} V_a^b(f_n)$. For $k$ sufficiently large we have $V_a^b(f_{n_k}) <  \liminf_{n \to \infty} V_a^b(f_n) + \varepsilon$ and
$\rho(f_{n_k}  (x_i),f^\infty(x_i)) \le \frac{\varepsilon}{2K}$, $i=0,1,\ldots,K$. Therefore,
	\begin{align*}
	& V(f^\infty, \chi) = \sum_{i=1}^{K} \rho\left(f^\infty(x_i), f^\infty(x_{i-1}) \right) \\
	& \le \sum_{i=1}^{K} \rho \left(f^\infty(x_i), f_{n_k}(x_{i}) \right) + \sum_{i=1}^{K} \rho(f_{n_k}(x_i),f_{n_k}(x_{i-1}))+ \sum_{i=1}^{K} \rho(f_{n_k}(x_{i-1}),f^\infty(x_{i-1})) \\
	& \le \frac{K\varepsilon}{2K} + \liminf_{n \to \infty} V_a^b(f_n) +\eps +\frac{K\varepsilon}{2K} 
	= \liminf_{n \to \infty} V_a^b(f_n) + 2\varepsilon.
	\end{align*}
	The first claim  follows by taking the supremum over all partitions, and the second claim is an easy consequence of the first one.
\end{proof}

In the next theorem we study sequences of functions which are equicontinuous from the left or from the right at a point.

\begin{theorem}\label{theorem_LimitFunc_LeftLocalModulu}
	Let $x^*\in (a,b]$, and $\{f_n\}_{n=1}^\infty$ be a sequence of functions $f_n: [a,b]\rightarrow X$ satisfying
$\omega^-(f_n, x^*, \delta) \le \omega(\delta)$, $0< \delta \le \delta_0$, $n\in\N $, where $\omega(\delta)$ is a modulus-bounding function. If $f^\infty = \lim\limits_{n\to \infty} f_n $ pointwisely on $[x^*-\delta_0, x^*]\cap [a,b]$, then
	$$
	\omega^-(f^\infty,x^*,\delta)\le \omega(\delta),\quad 0< \delta \le \delta_0.
	$$
	In particular, if $\lim\limits_{\delta \to 0^+} \omega(\delta) =0$ then  $f^\infty$ is left continuous at $x^*$.
\end{theorem}
\begin{proof}
Let $\delta \in (0,\delta_0]$. Fix $z\in [x^*-\delta, x^*]\cap[a,b]$. By the assumption, 
$$
\rho(f_n(z),f_n(x^*)) \le \omega^-(f_n,x^*,\delta) \le  \omega(\delta), \quad n\in\N.
$$
Let $\varepsilon>0$ be arbitrarily small. There exists $N(\varepsilon,z)$ such that
$$
\rho(f^\infty(z),f_{n}(z)) \le \frac{\varepsilon}{2}\quad \mbox{and} \quad \rho(f^\infty(x^*),f_{n}(x^*)) \le \frac{\varepsilon}{2}
$$ 
for all $n \ge N(\varepsilon,z)$. For such $n$ we have
\begin{align*}
& \rho(f^\infty(z),f^\infty(x^*)) \le  \rho(f^\infty(z),f_n(z)) + \rho(f_n(z),f_n(x^*)) + \rho(f_n(x^*),f^\infty(x^*)) \\
& \le \frac{\varepsilon}{2} + \omega(\delta) + \frac{\varepsilon}{2} = \varepsilon + \omega(\delta).
\end{align*}
Since $\varepsilon > 0$ was taken arbitrarily, it follows that  $\rho(f^\infty(z),f^\infty(x^*)) \le \omega(\delta)$. 
Thus,
$$
	\omega^-(f^\infty,x^*,\delta) = \sup \left\{\rho(f^\infty(z),f^\infty(x^*)) \  : \ z \in [x^*-\delta, x^*]\cap[a,b] \right\} \le \omega(\delta).
$$
In particular, it follows from Proposition~\ref{Prop_LocalModul_Vanishes} 
that $f^\infty$ is left continuous at $x^*$.
\end{proof}

\noindent An analogous result holds for the right continuity at $x^*$.
\medskip

Arguing along the same lines, one can also prove an analogous statement for the two-sided local modulus of continuity.

\begin{theorem}\label{Theorem_LimitFunc_LocalModulu}
	Let $x^*\in[a,b]$ and let $\, \{f_n\}_{n=1}^{\infty}$ be a sequence of functions $\, f_n :  [a,b] \to X$ satisfying  
	${\omega(f_n, x^*, \delta) \le \omega(\delta)}$,\, $0< \delta \le \delta_0$, $n \in \N$, where $\omega(\delta)$ is a modulus-bounding function. If $f^\infty = \lim\limits_{n \to \infty} f_n $ pointwisely on $[x^*-\frac{\delta_0}{2}, x^*+\frac{\delta_0}{2}]\cap [a,b]$, then
	$$
	\omega(f^\infty,x^*,\delta) \le \omega(\delta), \quad 0< \delta \le \delta_0.
	$$
	In particular, if $\lim\limits_{\delta \to 0^+} \omega(\delta) =0$ then $f^\infty$ is continuous at $x^*$.
\end{theorem}

As the last statement in this section, we formulate a property similar to Theorem~\ref{theorem_LimitFunc_LeftLocalModulu} for the local moduli of the function $v_f$.

\begin{propos}\label{Propos_Limit_v_LocalModulu}
	Let $x^*\in (a,b]$, and let and $\{f_n\}_{n=1}^\infty$ be a sequence of functions $f_n: [a,b]\rightarrow X$, $f_n \in \BV$,  satisfying
	${\omega^-(v_{f_n}, x^*, \delta) \le \omega(\delta)}$, $0< \delta  \le \delta_0$, $n\in\N $, where $\omega(\delta)$ is a modulus-bounding function. If $f^\infty = \lim\limits_{n\to \infty} f_n $ pointwisely on $[a,b]$, then
	$$
	\omega^-(v_{f^\infty},x^*,\delta)\le \omega(\delta),\quad 0< \delta \le \delta_0.
	$$
	In particular,   if $\lim\limits_{\delta \to 0^+} \omega(\delta) =0$  then  $v_{f^\infty}$ is left continuous at $x^*$. 
\end{propos}

\begin{proof}
	Let $x\in [x^*-\delta, x^*] \cap [a,b]$.
	By Theorem~\ref{theorem_LimitFunc_BV} and by the monotonicity of the variation function we have
	\begin{align*}
	v_{f^\infty}(x^*)-v_{f^\infty}(x) &=  V_x^{x^*} (f^\infty)\le \liminf_{n \to \infty} V_x^{x^*} (f_n) \le v_{f_n}(x^*)-v_{f_n}(x) \le \LeftLocalModul{v_{f_n}}{x^*}{\delta} \le \omega(\delta).
	\end{align*}
	Taking supremum over $x \in [x^*-\delta, x^*]\cap[a,b]$ we get the first claim. The second claim follows from Proposition~\ref{Prop_LocalModul_Vanishes}.
\end{proof}

Analogous statements hold for the right local modulus of continuity and for the two-sided local modulus of continuity.


\section {Multifunctions, their chain functions and metric selections}\label{Sect_MetricSelections}

The main object of this paper are  set-valued functions (SVFs, multifunctions) mapping $[a,b]$ to $\Comp$. First we recall some basic notions on such SVFs.

The graph of a multifunction $F$ is the set of points in $\R^{d+1}$ defined as
$$
\Graph(F)= \left \{(x,y) \ : \ y\in F(x),\; x \in [a,b] \right \}.
$$
It is easy to see that if $F \in  \BV$ then $\Graph(F)$ is a bounded set and $F$ has a bounded range, namely $\|F\|_\infty = \left| \bigcup_{x\in[a,b]} F(x) \right| <  \infty $.
We denote the class of SVFs of bounded variation with compact graphs by~$\calF$. 

For a set-valued function $F : [a,b] \to \Comp$, a single-valued function ${s:[a,b] \to \Rd}$ such that $s(x) \in F(x)$ for all $x \in [a,b]$ is called a selection of~$F$.

Below we present some definitions and results from~\cite{DFM:MetricIntegral} that will be used in this paper. In particular, we recall the definitions of chain functions and metric selections. 

Given a multifunction $F: [a,b] \to \Comp$, a partition $\chi=\{x_0,\ldots,x_n\} \subset [a,b]$, $a=x_0 < \cdots < x_n=b$, and a corresponding metric chain $\phi=(y_0,\ldots,y_n) \in \CH \left ( F(x_0),\ldots,F(x_n) \right )$ (see Definition~\ref{Def_MetChain_MetSelection}), the \textbf{chain function} based on $\chi$ and $\phi$ is
\begin{equation}\label{def_ChainFunc}
    c_{\chi, \phi}(x)= \left \{ \begin{array}{ll}
                        y_i, & x \in [x_i,x_{i+1}), \quad i=0,\ldots,n-1,\, \\
                        y_n, & x=x_n.
                    \end{array}
           \right.
\end{equation}

\begin{result}\label{Result_ChainFunct_Bounded&BV} \cite{DFM:MetricIntegral}\,  For $F \in \calF$, all chain functions satisfy $V_a^b(c_{\chi, \phi}) \le V_a^b(F)$ and $\|c_{\chi, \phi}\|_\infty \le\|F\|_\infty$.
\end{result}

A selection $s$ of $F$ is called a \textbf{metric selection}, if there is a sequence of chain functions $\{ c_{\chi_k, \phi_k} \}_{k \in \N}$ of~$F$ with ${\lim_{k \to \infty} |\chi_k| =0}$ such that
$$
s(x)=\lim_{k\to \infty} c_{\chi_k, \phi_k}(x) \quad \mbox{pointwisely on} \ [a,b].
$$
We denote the set of all metric selections of $F$ by $\mathcal{S}(F)$.

Note that the definitions of chain functions and metric selections imply that a metric selection $s$ of a multifunction $F$ is constant in any open interval where the graph of $s$ stays in the interior of  $\Graph(F)$.

\begin{result}\label{Result_MetSel_ThroughAnyPoint_Repres}\cite{DFM:MetricIntegral} \,
    Let $F \in \calF$. Through any point $\alpha \in \Graph(F)$ there exists a metric selection which we denote by~${\,s_\alpha}$.
    Moreover,  $F$ has a representation by metric selections, namely
$$
    F(x) = \{ s_\alpha(x) \ :\ \alpha \in \Graph(F)\}.
$$
\end{result}

\begin{result}\label{Result_MetrSel_InheritVariation}\cite {DFM:MetricIntegral}\,
    Let $s$ be a metric selection of $F \in \calF$. Then $V_a^b(s) \le V_a^b(F)$ and $\|s\|_\infty\le \|F\|_\infty$.
\end{result}

The next statements focus on local regularity properties of chain functions and metric selections. They refine results in~\cite{DFM:Book_SV-Approx} and~\cite{DFM:MetricIntegral}.

\begin{lemma}\label{lemma_LeftLocModuli_c_n<=LocModuli_v_f}
Let $F\in \calF$ and let $c_{\chi, \phi}$ be a chain function corresponding to a partition~$\chi$ and a metric chain~$\phi$ as in~\eqref{def_ChainFunc}. Then for any $x^*\in[a,b]$ we have
$$
\omega^{-}(c_{\chi, \phi},x^*,\delta) \le \omega^{-}(v_F,x^*,\delta+|\chi|),\quad \delta>0.
$$
\end{lemma}

\begin{proof}
The claim holds trivially for $x^* = a$. So we assume that $x^* \in (a,b]$. Let $\chi=\{x_0,\ldots,x_n\}$, $a=x_0 < \cdots < x_n=b$. 
We have  $x^*\in[x_k,x_{k+1})$ for some $0 \le k \le n-1$ or $x^*=x_n=b$. Take $z \in [a,b]$ such that $x^*-\delta \le z \le x^*$. If $x_k \le z \le x^*$, then  $c_{\chi, \phi}(z) = c_{\chi, \phi}(x^*) = c_{\chi, \phi}(x_k)$, and thus $|c_{\chi, \phi}(x^*)-c_{\chi, \phi}(z)| = 0$.
Otherwise there is $i< k$  such that $x_i \le z <x_{i+1}$. By the definitions of the chain function and of the metric chain we get
\begin{align*}
	|c_{\chi, \phi}(x^*)-c_{\chi, \phi}(z)| &= |c_{\chi, \phi}(x_k)-c_{\chi, \phi}(x_i)| \le \sum_{j=i}^{k-1} |c_{\chi, \phi}(x_{j+1})-c_{\chi, \phi}(x_j)| \le	\sum_{j=i}^{k-1} \haus \big( F(x_{j+1}),F(x_j) \big).
\end{align*}
Using the definitions of the variation of $F$, of $v_F$ and of $\omega^-$, we continue the estimate:
\begin{align*}
|c_{\chi, \phi}(x^*)-c_{\chi, \phi}(z)| \le  V_{x_i}^{x_k}(F) \le V_{x_i}^{x^*}(F) = v_F(x^*)-v_F(x_i) \le \omega^{-}(v_F,x^*,x^*-x_i) \le \omega^{-}(v_F,x^*,\delta+|\chi|).
\end{align*}
Taking the supremum over $z \in [x^*-\delta, x^*]\cap[a,b]$ we obtain the claim of the lemma.
\end{proof}

\begin{lemma}\label{lemma_RightLocModuli_c_n<=LocModuli_v_f}
	Let $F\in \calF$ and let $c_{\chi, \phi}$ be a chain function corresponding to a partition~$\chi$ and a metric chain~$\phi$. Then for any $x\in[a,b]$ we have
	$$
	\omega^{+}(c_{\chi, \phi},x^*,\delta) \le 2\omega\left(v_F,x^*,2(\delta+|\chi|) \right),\quad \delta>0.
	$$
\end{lemma}

\begin{proof}
If $x^* = b$, then the claim holds trivially. So we assume that $x^* \in [a,b)$. Let $x^*\in[x_k,x_{k+1})$ for some $0 \le k \le n-1$.
Take $z \in [a,b]$ such that $x^* \le z \le x^*+\delta$. There is $i\ge k$ such that $x_i \le z <x_{i+1}$. By the definition of the chain function we get
\begin{align*}
& |c_{\chi, \phi}(x^*)-c_{\chi, \phi}(z)| = |c_{\chi, \phi}(x_k)-c_{\chi, \phi}(x_i)| \le \sum_{j=k}^{i-1} |c_{\chi, \phi}(x_{j+1})-c_{\chi, \phi}(x_j)| \\
& \le \sum_{j=k}^{i-1} \haus (F(x_{j+1}),F(x_j)) \le  V_{x_k}^{x_i}(F). 
\end{align*}
Using the definitions of  the variation of $F$, of the variation function $v_F$ and~\eqref{Def_ClassicalModuliContin},~\eqref{defin_LeftModul},~\eqref{defin_RightModul},~\eqref{LocalModules-Prop1}, we obtain
\begin{align*}
|c_{\chi, \phi}(x^*)-c_{\chi, \phi}(z)| &\le  V_{x_k}^{x_i}(F) \le V_{x_k}^{x^*}(F) +  V_{x^*}^{z}(F) \le \omega^{-}(v_F,x^*,|\chi|) + \omega^{+}(v_F,x^*,\delta) \\& \le \omega(v_F,x^*,2|\chi|) + \omega(v_F,x^*,2\delta) \le 2\omega \left( v_F,x^*,2(|\chi|+\delta ) \right).
\end{align*}
The claim of the lemma follows by taking the supremum over $z \in [x^*, x^*+\delta]\cap[a,b]$.
\end{proof}

\begin{lemma}\label{Lemma_LocModuli_c_n<=LocModuli_v_F}
	Let $F \in \calF$ and let $c_{\chi, \phi}$ be a chain function corresponding to a partition~$\chi$ and a metric chain~$\phi$.
Then for any $x^*\in[a,b]$ we have
	$$
	\LocalModulCont{ c_{\chi, \phi} }{x^*}{\delta} \le \LocalModulCont{v_F}{x^*}{\delta +2|\chi|}, \quad \delta > 0.
	$$
\end{lemma}
\begin{proof} Let $x,z \in [x^*-\delta/2, x^*+\delta/2] \cap [a,b]$, $x < z$. First assume that $z \ne x_n$.  In this case  there exist $k,i$ with $0 \le k \le i \le n-1$ such that $x\in[x_k,x_{k+1})$ and $z \in [x_i,x_{i+1})$. We get 
\begin{align*}
|c_{\chi, \phi}(x)-c_{\chi, \phi}(z)| &= |c_{\chi, \phi}(x_k)-c_{\chi, \phi}(x_i)| \le \sum_{j=k}^{i-1} |c_{\chi, \phi}(x_{j+1})-c_{\chi, \phi}(x_j)| \le	\sum_{j=k}^{i-1} \haus (F(x_{j+1}),F(x_j)) \\& \le  V_{x_k}^{x_i}(F) \le V_{x_k}^{z}(F) = v_F(x_k)-v_F(z) \le \LocalModulCont{v_F}{x^*}{\delta+ 2|\chi| }. 
\end{align*}	
The above inequalities hold also for $x<z=x_n$. In the case when $x=z$ this estimate is trivial.	
Taking the supremum over $x,z \in [x^*-\delta/2, x^*+\delta/2] \cap [a,b]$ we obtain $	\LocalModulCont{ c_{\chi, \phi} }{x^*}{\delta}\le \LocalModulCont{v_F}{x^*}{\delta+ 2|\chi|}$.
\end{proof}

\begin{theorem}\label{theorem_LocLeftModuli_s<=LocLeftModuli_v_f}
	Let $F\in \calF$,  $s$ be a metric selection of~$F$ and  $x^*\in[a,b]$. Then
	$$
	\omega^{-}(s,x^*,\delta) \le \omega^{-}(v_F,x^*,2\delta),\quad \delta>0.
	$$
	In particular, if $F$ is left continuous at $x^*$, then $s$ is left continuous at $x^*$.
\end{theorem}

\begin{proof}
	Let $s$ be a metric selection of~$F$. Then there exists a sequence of partitions $\{\chi_n\}_{n \in \N}$ with $|\chi_n| \to 0$, $n \to \infty$, and a corresponding sequence of chain functions $\{c_n\}_{n \in \N}$ such that $s(x)=\lim\limits_{n\rightarrow \infty}c_n(x)$ pointwisely for all $x\in[a,b]$. For $n$ so large that $|\chi_n|\le \delta$, we get by Lemma~\ref{lemma_LeftLocModuli_c_n<=LocModuli_v_f}
	$$
	\omega^{-}(c_n,x^*,\delta) \le \omega^{-}(v_F,x^*,\delta+|\chi_n|) \le \omega^{-}(v_F,x^*,2\delta).
	$$
Theorem~\ref{theorem_LimitFunc_LeftLocalModulu} implies
	\begin{align*}
	\omega^{-}(s,x^*,\delta) \le \omega^{-}(v_F,x^*,2\delta).
	\end{align*}
	Moreover, if $F$ is left continuous at $x^*$ then by Propositions~\ref{Prop_LeftContEquiv_v_f-f} and~\ref{Prop_LocalModul_Vanishes} we have $\omega^{-}(v_F,x^*,2\delta) \to 0$ as $\delta \to 0$. The latter implies that $s$ is left continuous at $x^*$.
\end{proof}

Using Lemma~\ref{lemma_RightLocModuli_c_n<=LocModuli_v_f} instead of Lemma~\ref{lemma_LeftLocModuli_c_n<=LocModuli_v_f} and arguing as above, we obtain

\begin{theorem}\label{theorem_LocRightModuli_s<=LocRightModuli_v_f}
	Let $F\in \calF$, $s$ be a metric selection of~$F$ and  $x^*\in[a,b]$. Then
	$$
	\omega^{+}(s,x^*,\delta) \le 2\omega(v_F,x^*,4\delta), \quad \delta>0.
	$$
\end{theorem}

Similarly, Lemma~\ref{Lemma_LocModuli_c_n<=LocModuli_v_F} and Theorem~\ref{Theorem_LimitFunc_LocalModulu} lead to 

\begin{theorem}\label{Theorem_LocModuli_s<=LocModuli_v_F}
	Let $F \in \calF$, $s$ be a metric selection of $F$  and $x^*\in[a,b]$. Then
	$$
	\LocalModulCont{s}{x^*}{\delta} \le \LocalModulCont{v_F}{x^*}{ 2\delta}, \quad \delta >0.
	$$
	In particular, if $F$ is continuous at $x^*$, then $s$ is continuous at $x^*$.
\end{theorem}

\begin{remark}\label{remark_ImproveModuliEstimates}
	Analysing the proofs, it is not difficult to see that the estimates in Theorems~\ref{theorem_LocLeftModuli_s<=LocLeftModuli_v_f}--\ref{Theorem_LocModuli_s<=LocModuli_v_F} can be improved  in the following way 
	$$
	\omega^{-}(s,x^*,\delta) \le \omega^{-}(v_F,x^*,\delta+\varepsilon), \quad \omega^{+}(s,x^*,\delta) \le 2\omega(v_F,x^*,2\delta+ \varepsilon), 
	\quad \LocalModulCont{s}{x^*}{\delta} \le \LocalModulCont{v_F}{x^*}{ \delta + \varepsilon}, 
	\quad \delta > 0,
	$$
	with an arbitrarily small $\varepsilon >0$. Taking the supremum of the both sides of the last inequality over $x^* \in [a,b]$ we obtain
	$$
	\ModulCont{s}{\delta} \le \ModulCont{v_F}{\delta+\varepsilon}.
	$$
	If $F\in \CBV$, then $v_F \in \CBV$ and $\omega(v_F, \delta)$  is continuous in $\delta$. Taking the limit as $\varepsilon \to 0+$ we get 
	$$
\ModulCont{s}{\delta} \le\ModulCont{v_F}{\delta}.
$$	
 Therefore also $s\in \CBV$.
\end{remark}

\begin{lemma}\label{lemma_Var(c_n)<LocalMod(v_F)}
Let $F\in \calF$ and let $c_{\chi, \phi}$ be a chain function corresponding to a partition $\chi$ and a metric chain~$\phi$. Let $\delta > 0$ be such that $[a+\delta+|\chi|, b-\delta] \ne \emptyset$. Then for any $x\in[a+\delta+|\chi|, b-\delta]$ we have
$$
V_{x-\delta}^{x+\delta}(c_{\chi, \phi}) \le V_{x-\delta-|\chi|}^{x+\delta}(F) \le \omega \left(v_F,x,2(\delta+|\chi|) \right).
$$
\end{lemma}

\begin{proof}
Let $\chi=\{x_0,\ldots,x_n\}$, $a=x_0 < \cdots < x_n=b$.
	By definition, $\left( c_{\chi, \phi}(x_j),c_{\chi, \phi}(x_{j+1}) \right) \in \Pair{F(x_j)}{F(x_{j+1})}$, $j=0,\ldots ,n-1$. Thus, $V_{x_i}^{x_k}(c_{\chi, \phi})\le V_{x_i}^{x_k}(F)$ for all $0\le i < k \le n$.
	If $x_k \le x-\delta < x+\delta  < x_{k+1}$, then $c_{\chi, \phi}(t) = c_{\chi, \phi}(x_k)$ for all $t \in [x-\delta, x+ \delta]$, and therefore $ V_{x-\delta}^{x+\delta}(c_{\chi, \phi})=0 $.
	In the case when $x_{i-1} \le x-\delta < x_i < \cdots < x_k \le x+\delta  < x_{k+1}$ we have   ${c_{\chi, \phi}(x-\delta)=c_{\chi, \phi}(x_{i-1})}$, ${c_{\chi, \phi}(x+\delta)=c_{\chi, \phi}(x_{k})}$. Thus, 
	$$
	V_{x-\delta}^{x+\delta}(c_{\chi, \phi})  =  V_{x_{i-1}}^{x_k}(c_{\chi, \phi})\le V_{x_{i-1}}^{x_k}(F) \le V_{x-\delta-|\chi|}^{x+\delta}(F).
	$$
For the second inequality, we continue the estimate as follows:
$$
V_{x-\delta}^{x+\delta}(c_{\chi, \phi})  \le V_{x-\delta-|\chi|}^{x+\delta}(F) = v_F(x + \delta) - v_F(x - \delta - |\chi|) \le \omega \left(v_F,x,2(\delta+|\chi|) \right).
$$
\end{proof}

\begin{theorem}\label{theorem_Var(s)<LocalMod(v_F)}
	Let $F\in \calF$ and let $s$ be a metric selection of~$F$. Then for all small $\delta > 0$ and all $x\in[a+2\delta, b-\delta]$ we have
	$$
	V_{x-\delta}^{x+\delta}(s) \le V_{x-2\delta}^{x+\delta}(F) \le \omega\left (v_F,x,4\delta \right).
	$$
\end{theorem}

\begin{proof}
Since $s$ is a metric selection, there exists a sequence of partitions $\{\chi_n \}_{n \in \N}$ with $|\chi_n| \to 0$, $n \to \infty$, and a corresponding sequence of chain functions $\{c_n\}_{n \in \N}$ such that $s(x)=\lim\limits_{n\to \infty}c_n(x)$ pointwisely. Take $n$ so large that $|\chi_n| < \delta$, then by Lemma~\ref{lemma_Var(c_n)<LocalMod(v_F)}   we have $ { V_{x-\delta}^{x+\delta}(c_n) \le V_{x-\delta-|\chi_n|}^{x+\delta}(F) \le V_{x-2\delta}^{x+\delta}(F) } $. In view of Theorem~\ref{theorem_LimitFunc_BV} we get
$V_{x-\delta}^{x+\delta}(s) \le V_{x-2\delta}^{x+\delta}(F) \le \omega\left (v_F,x,4\delta \right) $.
\end{proof}

The statement of Theorem~\ref{theorem_Var(s)<LocalMod(v_F)} can be improved in the same manner like in Remark~\ref{remark_ImproveModuliEstimates}. Namely,  the estimate
$$
	V_{x-\delta}^{x+\delta}(s) \le V_{x-\delta-\varepsilon}^{x+\delta}(F) \le \omega \left (v_F,x,2\delta+ \varepsilon \right)
$$
holds with an arbitrarily small $\varepsilon > 0$.

The next result was announced in~\cite[Lemma~3.9]{DFM:MetricIntegral} without a detailed proof. Although the result is intuitively clear, its proof is rather complicated. We present the full proof in Appendix A. 

\begin{theorem}\label{theorem_PW-limit_of_MS_isMS}
	For $F\in \calF$, the pointwise limit of a sequence of metric selections of~$F$ is a metric selection of $F$.
\end{theorem}


\section {Weighted metric integral}\label{Sect_MetricIntegral}

The well-known Aumann integral~\cite{AUM:65} of a multifunction~$F$ is defined as
\begin{equation}\label{Def_AumannIntegral}
\int_a^b F(x)dx = \left\{\int_a^b s(x)dx \ : \ s\ \mbox{is an integrable selection of}\ F \right\}.
\end{equation}
Everywhere in this context we understand the integral of a function $f : [a,b] \to \Rd$ to be applied to each component of $f$.

It is known that the Aumann integral is convex for each function $F \in \calF$, even if the values of $F$ are not convex. Moreover, 
\begin{equation}\label{Prop_AumannIntegral}
\int_a^b F(x)dx = \int_a^b \co \big( F(x) \big) dx, \quad \int_a^b w(x)Adx =\left (\int_a^b w(x)dx \right ) \; \co(A),
\end{equation}
where $A \in \Comp$ and $w(x) \ge 0$, $x \in [a,b]$.

The metric integral of SVFs has been introduced in~\cite{DFM:MetricIntegral}. In contrast to the Aumann integral, the metric integral is free of the undesired effect of the convexification. We recall its definition. First we define the metric Riemann sums. For a multifunction $\Map{[a,b]}$ and for a partition $\chi=\{x_0,\ldots,x_n\}$,  ${a=x_0<x_1<\cdots < x_n=b}$, the metric Riemann sum of $F$ is defined by
	$$
    {\scriptstyle(\cal M)} S_{\chi} F = \bigoplus_{i=0}^{n-1}(x_{i+1}-x_i)F(x_i).
	$$

\begin{defin}\label{Def_MetricIntegral} \cite{DFM:MetricIntegral}
The metric integral of $F$ is defined as the Kuratowski upper limit of metric Riemann sums corresponding to partitions with norms tending to zero, namely, 
$$
\MetInt =  \limsup_{|\chi| \to 0} {\scriptstyle(\cal M)} S_{\chi} F.
$$
The upper limit here is understood in the following sense: $y \in  \limsup_{|\chi| \to 0} {\scriptstyle(\cal M)} S_{\chi} F$ if there is a sequence of partitions $\{ \chi_n\}_{n \in \N}$ with $|\chi_n| \to 0$, $n \to \infty$, and a sequence $\{y_n\}_{n \in \N}$ such that $y_n \in {\scriptstyle(\cal M)} S_{\chi_n} F$ and $y_n \to y$, $n \to \infty$.
\end{defin}
It is easy to see that the set $\MetInt$ is non-empty if $F$ has a bounded range.

The following result from~\cite{DFM:MetricIntegral}  relates the metric integral of $F\in \calF$ to its metric selections.
\begin{result}\label{Result_MetrInt=IntOfMetrSel} \cite{DFM:MetricIntegral}
	Let $F \in \calF$. Then
	$
		\MetInt=\left \{ \IntSel {s} \ :\ s \in\mathcal{S}(F) \right \}.
	$
\end{result}

In this section we define an extension of the metric integral, namely,  the weighted metric integral. 

For a set-valued function $\Map{[a,b]}$, a weight function $k:[a,b] \to \R$ and for a partition $\chi=\{x_0,\ldots,x_n\}$, ${a=x_0<x_1<\cdots < x_n=b}$, we define the\textbf{ weighted metric Riemann sum} of $F$ by

\begin{align*}
{\scriptstyle(\cal M_k)} S_{\chi} F &= \left\{ \sum_{i=0}^{n-1} (x_{i+1}-x_i)k(x_i) y_i \ : \ (y_0,\ldots,y_{n-1}) \in \CH(F(x_0),\ldots,F(x_{n-1})) \right\} \\
&= \bigoplus_{i=0}^{n-1}(x_{i+1}-x_i)k(x_i) F(x_i).
\end{align*}

\begin{remark}\label{Remark_WeightedMetrSum}
	The elements of $ {\scriptstyle(\cal M_k)} S_{\chi} F $ are of the form $ \int_{a}^{b} k_\chi(x) c_{\chi,\phi}(x) dx $, where $c_{\chi, \phi}$ is a chain function based on the partition $\chi$ and a  metric chain $\phi  = (y_0, \ldots, y_n) \in CH\left (F(x_0), \ldots,F(x_{n})\right )$, and $k_\chi$ the piecewise constant function defined by
	\begin{equation}\label{weight-step-function}
	k_\chi(x)= \left \{ \begin{array}{ll}
	k(x_i), & x \in [x_i,x_{i+1}), \quad i=0,\ldots,n-1,\, \\
	k(x_n), & x=x_n.
	\end{array}
	\right.
	\end{equation}
\end{remark}

We define the weighted metric integral of $F$ as the Kuratowski upper limit of weighted metric Riemann sums.
\begin{defin}\label{Defin_WeightedMetricInt}
	The \textbf{weighted metric integral} of $F$ with the weight function  $k$ is defined by
	$$
	\WeightedMetInt  = \limsup_{|\chi| \to 0} {\scriptstyle(\cal M_k)} S_{\chi} F.
	$$
\end{defin}

The set $\WeightedMetInt $ is non-empty whenever the SVF $kF$ has a bounded range.

Observe that the weighted metric integral of $F$ with the weight $k$ is not the metric integral of the multifunction $kF$. The difference is that the metric chains in Definition~\ref{Defin_WeightedMetricInt} are constructed on the base of the function $F$, and not $kF$ which would be in the latter case. 

In the remaining part of this section we extend results obtained for the metric integral in~\cite{DFM:MetricIntegral} to the weighted metric integral.

\begin{remark}	
It is possible to define a ``right'' weighted metric Riemann sum as
$$
{\scriptstyle(\cal M_k)} \widetilde{S}_{\chi} F = \bigoplus_{i=0}^{n-1}(x_{i+1}-x_i)k(x_{i+1}) F(x_{i+1}),
$$
and a corresponding weighted metric integral. For BV functions $F$ and $k$, this integral is identical with \linebreak $\WeightedMetInt$.   This can be concluded from the following lemma.
\end{remark}

\begin{lemma}\label{Lemma_dist_RiemannSum}
Let $F, k \in \BV$. Then
$$
\haus \left ( {\scriptstyle(\cal M_k)} \widetilde{S}_{\chi} F,{\scriptstyle(\cal M_k)} S_{\chi} F \right ) 
\le |\chi| \left ( \|k\|_\infty \, V_a^b(F) + \|F\|_\infty \, V_a^b(k) \right ).
$$
\end{lemma}

\begin{proof}
Fix a partition $\chi$ and consider a corresponding chain ${\phi=(y_0,\ldots,y_n)\in CH(F(x_0), \ldots ,F(x_n) ) }$. 
We have
\begin{align*}
& \haus \left ( {\scriptstyle(\cal M_k)} \widetilde{S}_{\chi} F, {\scriptstyle(\cal M_k)} S_{\chi} F \right ) 
\\& \le \sup \left \{ \left |\sum_{i=0}^{n-1} k(x_{i+1})y_{i+1}(x_{i+1}-x_i)-\sum_{i=0}^{n-1} k(x_i)y_i(x_{i+1}-x_i)\right | \ : \ \phi \in CH(F(x_0), \ldots,F(x_n))  \right \}
\\& \le \sup \left \{ \sum_{i=0}^{n-1} \left |k(x_{i+1})y_{i+1}-k(x_i)y_i \right |(x_{i+1}-x_i) \ : \ \phi \in CH(F(x_0), \ldots ,F(x_n))  \right \}.
\end{align*}
Since
\begin{align*}
|k(x_{i+1})y_{i+1}-k(x_i)y_i| &\le |k(x_{i+1})y_{i+1}-k(x_{i+1})y_i|+|k(x_{i+1})y_i-k(x_i)y_i| \\& \le \|k\|_\infty \, \haus(F(x_{i+1}),F(x_i)) + \|F\|_\infty\, |k(x_{i+1})-k(x_i)|,
\end{align*}
the desired estimate follows.
\end{proof}

The next theorem is an extension of Result~\ref{Result_MetrInt=IntOfMetrSel} to the weighted metric integral.

\begin{theorem}\label{Theo_W-MetrInt=W-IntOfMetrSel}
	Let $F \in \cal{F}[a,b]$ and  $k \in \BV$. Then
	$$
	\WeightedMetInt=\left \{ \int_{a}^{b}k(x)s(x)dx \ : \ s \in\mathcal{S}(F) \right \}.
	$$
\end{theorem}

\begin{proof}
	By Result~\ref{Result_MetrSel_InheritVariation}, every metric selection $s$ of $F \in \cal{F}[a,b]$ is BV, and thus $ks$ is Riemann integrable. Denote $I=\left \{ \int_{a}^{b}k(x)s(x)dx \ : s \in\mathcal{S}(F) \right \}$. 

	We first show that $I \subseteq \WeightedMetInt$.
	Let $s$ be a metric selection of $F$. Then $s$ is the pointwise limit of a sequence of chain functions $\{c_n\}_{n\in \N}$ corresponding to
	partitions $\{\chi_n\}_{n\in \N}$ with $\lim_{n \to \infty} |\chi_n| =0$. Denote $k_n=k_{\chi_n}$~(see~\eqref{weight-step-function}) and $\sigma_n = \int_a^b k_n(x)c_n(x)dx$. By Remark~\ref{Remark_WeightedMetrSum}, $\sigma_n \in  {\scriptstyle(\cal M_k)} S_{\chi_n} F$.
	
Clearly, $\|k_n\|_\infty \le \|k\|_\infty$ and $V_a^b(k_n) \le V_a^b(k)$. By Helly's Selection Principle there exists a subsequence $\{k_{n_\ell}\}_{\ell \in \N}$ that converges pointwisely to a certain function $k^*$. For simplicity we denote this sequence by $\{k_n\}_{n \in \N}$ again.  It is easy to see that $k^*(x)=k(x)$ at all points of continuity of $k$. Indeed, for a partition $\chi_n$ there is an index $i_n$ such that $x\in[x_{i_n},x_{i_n+1})$, where $x_{i_n}$ and $x_{i_n+1}$ are subsequent points  in $\chi_n$.  By~\eqref{weight-step-function} we get
	$$
	|k_n(x)-k(x)|=|k_n(x_{i_n})-k(x)| = |k(x_{i_n})-k(x)| \le \LocalModulCont{k}{x}{|\chi_n|}.
	$$	
	Thus $\lim_{n \to \infty} k_n(x)c_n(x) = k(x)s(x)$ at all points of continuity of $k$.
	Note that since $k$ is BV, it has at most countably many points of discontinuity in $[a,b]$. By Result~\ref{Result_ChainFunct_Bounded&BV} and the Lebesgue Dominated Convergence Theorem we obtain
	$$
	\int_{a}^{b}k(x)s(x)dx = \lim_{n \to \infty} \int_{a}^{b}k_{n}(x)c_{n}(x)dx = \lim_{n \to \infty} \sigma_{n} \in \WeightedMetInt.
	$$

	It remains to show the converse inclusion  $\WeightedMetInt \subseteq I$. Let ${\sigma \in \WeightedMetInt}$. There exists a sequence $\{\sigma_n\}_{n\in \N}$, ${ \sigma_n \in  {\scriptstyle(\cal M_k)} S_{\chi_n} F }$, such that
	${\displaystyle \sigma = \lim_{n\to\infty} \sigma_n  }$. By Remark~\ref{Remark_WeightedMetrSum} we have $\sigma_n = \int_{a}^{b}k_n(x)c_n(x)dx$. Applying Helly's Selection Principle two times consequently, we conclude that there is a  subsequence $\{ k_{n_\ell} \}_{\ell \in \N}$ that converges pointwisely to a certain function  $k^*$,  and then there is a subsequence $\left\{ c_{n_{\ell_m}} \right\}_{m \in \N}$ that converges pointwisely to a certain function $s$. By definition, $s\in\mathcal{S}(F)$. It follows  from Result~\ref{Result_ChainFunct_Bounded&BV} and the Lebesgue Dominated Convergence Theorem that 
	$$
	\sigma = \lim_{ m \to \infty } \sigma_{n_{l_m}} = \lim_{ m \to \infty } \int_{a}^{b}k_{n_{l_m}}(x)c_{n_{l_m}}(x)dx= \int_{a}^{b}k^*(x)s(x)dx=\int_{a}^{b}k(x)s(x)dx,
	$$
	which completes the proof.
\end{proof}

Theorem~\ref{Theo_W-MetrInt=W-IntOfMetrSel}, \eqref{Def_AumannIntegral}  and~\eqref{Prop_AumannIntegral} yield the following statement.

\begin{corol}\label{Corol_MetInt_included_AumannInt}
	Under the assumptions of Theorem~\ref{Theo_W-MetrInt=W-IntOfMetrSel} we have
\begin{equation}\label{WeightedMetrInt_subset_WeigtedAumann_Int}
\WeightedMetInt \subseteq \int_{a}^{b}k(x)F(x)dx.
\end{equation}
Moreover,
$$
\co\left(\WeightedMetInt\right ) \subseteq \int_{a}^{b}k(x)F(x)dx.
$$
\end{corol}

Corollary~\ref{Corol_MetInt_included_AumannInt} implies  the following ``inclusion property'' of the weighted metric integral as stated below.
\begin{propos}\label{Propos_InclusionProperty}
	For $F \in \cal{F}[a,b]$ and $k \in \BV$ we have
	\begin{equation}\label{Prop1}
	\int_{a}^{b}k(x)dx \left ( \bigcap_{x\in [a,b]}F(x) \right ) \subseteq \WeightedMetInt \subseteq (b-a) \, \co  \left ( \bigcup_{x\in [a,b]} k(x)F(x)  \right ).
	\end{equation}
	Moreover, if $k(x) \ge 0$ , $x\in[a,b]$ and $\int_{a}^{b}k(x)dx \neq 0$ then
	\begin{equation}\label{Prop2}
 \bigcap_{x\in [a,b]}F(x) \subseteq \frac{\WeightedMetInt}{\int_{a}^{b}k(x)dx}  \subseteq  \co \left( \bigcup_{x\in [a,b]} F(x) \right ).
\end{equation}
\end{propos}

\begin{proof}
	First we prove the left inclusion in~\eqref{Prop1}. If $\bigcap_{x\in [a,b]}F(x) = \emptyset$ then there is nothing to prove. Suppose  $\bigcap_{x\in [a,b]}F(x) \ne \emptyset$.
 Let $ p \in \bigcap_{x\in [a,b]}F(x)$. Then $s(x)\equiv p$, $x\in[a,b]$, is a metric selection of $F$, since for any partition $\chi$ the function $c_{\chi,\phi}(x) \equiv p$ is a chain function corresponding to the chain $\phi=(p,\ldots,p)$. Therefore,
	$$
	p\int_{a}^{b}k(x)dx = \int_{a}^{b}k(x)s(x)dx \in \WeightedMetInt .
	$$

	To show the right inclusion in~\eqref{Prop1}, we use~\eqref{WeightedMetrInt_subset_WeigtedAumann_Int} and~\eqref{Prop_AumannIntegral} and write
	$$
\WeightedMetInt \subseteq \int_{a}^{b}k(x)F(x)dx \subseteq \int_{a}^{b}\left( \bigcup_{x\in [a,b]} k(x)F(x) \right) dx = (b-a)\, \co \left ( \bigcup_{x\in [a,b]} k(x)F(x) \right ).
$$ 
In the case when $k(x) \ge 0$ and $\int_{a}^{b}k(x)dx \neq 0$, the left inclusion in~\eqref{Prop2} follows directly from~\eqref{Prop1}. To prove the right inclusion in~\eqref{Prop2}, we start with~\eqref{WeightedMetrInt_subset_WeigtedAumann_Int}.
Denoting $R=\bigcup_{x\in [a,b]} F(x) \in \Comp$  we get in view of the second property in~\eqref{Prop_AumannIntegral}
$$
\WeightedMetInt \subseteq \int_{a}^{b}k(x)F(x)dx \subseteq \int_{a}^{b}k(x)Rdx =  \left( \int_{a}^{b} k(x) dx \right )\; \co(R),
$$
and the right inclusion follows.

\end{proof}

Note that the middle set in~\eqref{Prop2} is a weighted average of $F(x)$ on $[a,b]$. Proposition~\ref{Propos_InclusionProperty}  says that it contains the intersection of the sets $\{F(x)\}_{x\in[a,b]}$ and is contained in the convex hull of their union.


\begin{propos}\label{WeihtedMetInt_is_compact}
Let $F \in \cal{F}[a,b]$ and $k \in \BV$. The set $\WeightedMetInt$ is compact.
\end{propos}

\begin{proof}
	Since $F$ and $k$ are both bounded, the set $\WeightedMetInt$ is bounded. To prove the proposition, it suffices to show that it is closed. Consider a convergent sequence $\{v_n\}_{n \in \N} \subset \WeightedMetInt$.  Let $\displaystyle \ v=\lim_{ n \to \infty }v_n $. By Theorem~\ref{Theo_W-MetrInt=W-IntOfMetrSel} we have $v_n=\int_{a}^{b}k(x)s_n(x)dx$ for some $s_n \in \mathcal{S}(F)$. The sequence $\{s_n\}_{n\in \N}$ is uniformly bounded and of uniformly bounded variation. By Helly's Selection Principle there exists a subsequence $\{s_{n_\ell}\}_{\ell \in \N}$ which converges pointwisely to a certain function $s^\infty$ as $\ell \to \infty$.  By Theorem~\ref{theorem_PW-limit_of_MS_isMS}, $s^\infty$ is a metric selection.  Clearly, $\lim_{\ell\to \infty} k(x)s_{n_\ell}(x) = k(x)s^\infty(x)$ pointwisely. 
Applying the Lebesgue Dominated Convergence Theorem we get 
	$$
	\int_{a}^{b}k(x)s^\infty(x)dx = \lim_{\ell \to \infty} \int_{a}^{b}k(x)s_{n_\ell}(x)dx = \lim_{\ell \to \infty}v_{n_\ell}=v,
			$$
	and thus $v \in \WeightedMetInt$.
\end{proof}


\section {The metric Fourier approximation of SVFs of bounded variation }\label{Sect_Fourier}

\subsection{On Fourier approximation of real-valued functions of bounded variation} \label{Subsec_FourierRealFunctions}
First we present the classical material relevant to our study  of SVFs.

For a $2\pi$-periodic  real-valued function $f : \RR \to \RR$ which is integrable over the period, its Fourier series is
$$
f(x) \sim \frac{1}{2} a_0 + \sum_{k=1}^\infty (a_k \cos{kx} + b_k \sin{kx}), 
$$
where
\begin{equation} \label{Fourier-a,b-coeff}
a_k =  a_k(f) = \frac{1}{\pi} \int_{-\pi}^{\pi} f(t) \cos{kt} dt, \quad k = 0,1,\ldots, 
\quad \mbox{and}
\quad b_k = b_k(f) = \frac{1}{\pi} \int_{-\pi}^{\pi} f(t) \sin{kt} dt, \quad k = 1,2,\ldots.
\end{equation}
Following the classical theory of Fourier series, we introduce the Dirichlet kernel (see e.g. \cite[Chapter  II]{Zygmund:TrigSeries})
$$
D_n(x) = \frac{1}{2} + \sum_{k=1}^n \cos{kx} = \frac{\sin{\left( n + \frac{1}{2} \right)x } }{ 2 \sin{ \left(\frac{1}{2}x \right) }} , \quad x \in \RR.
$$
For the partial sums of the Fourier series one has the well-known representation 
\begin{equation} \label{Dirichlet-repr}
\mathscr{S}_n f (x) =  \frac{1}{2} a_0 + \sum_{k=1}^n (a_k \cos{kx} + b_k \sin{kx}) = \frac{1}{\pi} \int_{-\pi}^{\pi} D_n(x-t)  f(t) dt = \frac{1}{\pi} \int_{-\pi}^{\pi} \partial_{n,x}(t)  f(t) dt,
\end{equation}
where $\partial_{n,x}(t) = D_n(x-t)$.

A basic result on the convergence of Fourier series of real-valued functions of bounded variation is the Dirichlet-Jordan Theorem (e.g., \cite[Chapter  II, (8.1) Theorem]{Zygmund:TrigSeries}).
\smallskip

\noindent \textbf{Dirichlet-Jordan Theorem}. \textit{ Let $f : \RR \to \RR$ be a $2\pi$-periodic function of bounded variation on $[-\pi,\pi]$. Then at every point $x$
	$$
	\lim_{n \to \infty} {\mathscr{S}_nf(x)} =  \frac{1}{2}( f(x-0) + f(x+0)).
	$$ 
	In particular, $\mathscr{S}_nf$ converges to $f$ at every point of continuity of $f$. If $f$ is continuous at every point of a closed interval $I$, then the convergence is uniform in $I$.
}
\medskip

Following~\cite[Chapter  II]{Zygmund:TrigSeries}, we introduce the so-called modified Dirichlet kernel
\begin{equation}\label{Formule_ModifKernel}
D^*_n(x) = \frac{1}{2} + \sum_{k=1}^{n-1} \cos{kx} + \frac{1}{2} \cos{nx} = \frac{1}{2} \sin{nx} \cot{ \left( \frac{1}{2}x \right)} , \quad x \in \RR,
\end{equation}
and the modified Fourier sum
$$
\mathscr{S}^*_n f (x) =  \frac{1}{\pi} \int_{-\pi}^{\pi} D^*_n(x-t)  f(t) dt.
$$
Clearly,
\begin{equation} \label{D-D*}
D_n(x) - D^*_n(x) = \frac{1}{2} \cos{nx},
\end{equation}
and
\begin{equation} \label{D-int}
\frac{1}{\pi} \int_{-\pi}^{\pi} D_n(x) dx = \frac{1}{\pi} \int_{-\pi}^{\pi} D^*_n(x) dx = 1.
\end{equation}

We also need the next result that follows immediately from~\cite[Chapter  II, (4.12) Theorem]{Zygmund:TrigSeries}.
\begin{lemma} \label{Lemma:BV_Estimate_ab} 
	Let $f \in \mathrm{BV}[-\pi,\pi]$. Then its Fourier coefficients \eqref{Fourier-a,b-coeff} 
	satisfy the estimate
	$$
	|a_n(f)| \le\frac{2V_{-\pi}^{\pi}(f)}{\pi n}, \quad |b_n(f)| \le\frac{2V_{-\pi}^{\pi}(f)}{\pi n}, \quad n \in \N.
	$$
\end{lemma}
A further property of the kernel $D^*_n$ which can be found in~\cite[Chapter  II, (8.2) Lemma]{Zygmund:TrigSeries} is 
\begin{lemma} \label{Lemma:Int_D_Estimate}
	There is a constant $C > 0$ such that for all $\xi \in [0,\pi]$ and all $n \in \N$
	\begin{equation} \label{Int_D_Estimate}
	\left| \frac{2}{\pi} \int_0^\xi D^*_n(x) dx \right| \le C.
	\end{equation}
\end{lemma}

\begin{remark}\label{Remark_Value_C} 
Analyzing the proof of  this statement in~\cite[Chapter  II, (8.2) Lemma]{Zygmund:TrigSeries}, one can see that one can take $C =2 $, i.e.
	$$
	\left| \frac{2}{\pi} \int_0^\xi D^*_n(x) dx \right| \le 2, \quad \xi \in [0,\pi], \quad n \in \N.
	$$
\end{remark}

\subsection{Extension to special classes of real-valued functions of bounded variation} \label{Subsec_FourierClasses}

It is known that functions of bounded variation with values in an arbitrary complete metric space $(X,\rho)$ are not necessarily continuous, but have right and left limits at any point~\cite{Chistyakov:On_BV-mappings}. To study such functions we introduce the left and right local quasi-moduli for discontinuous functions of bounded variation.

\begin{defin}\label{Def_NewLocalModuli}
	For a function $f : [a,b] \to X$ of bounded variation and $x^* \in (a,b]$ we define the left local quasi-modulus
	$$
	\NewLeftLocalModul{f}{x^*}{\delta} = \sup{ \big \{ \rho(f(x^*-0),f(x)) \ : \ x \in [x^*-\delta,x^*) \cap [a,b] \big\} }, \quad \delta >0,
	$$
	and for  $x^* \in [a,b)$ the right local quasi-modulus
	$$
	\NewRightLocalModul{f}{x^*}{\delta} = \sup{ \{ \rho(f(x^*+0),f(x)) \ : \ x \in (x^*,x^* + \delta] \cap [a,b] \} }, \quad \delta >0,
	$$
	where $\displaystyle f(x-0) = \lim_{ t \to x-0 }f(t)$, $\displaystyle f(x+0) = \lim_{ t \to x+0 }f(t)$.
\end{defin}

The facts given in the following remark are direct consequences of the above definitions.
 
\begin{remark}\label{Remark_NewModuli} 
Let $f : [a,b] \to X$ be a BV function and $x^* \in (a,b]$ for the left modulus or $x^* \in [a,b)$ for the right modulus, respectively.
	\begin{enumerate}
		\item[(i)] If $f$ is monotone then
		$$ \NewLeftLocalModul{f}{x^*}{\delta} = \rho(f(x^*-0), f(x^* - \delta)),\quad
		\NewRightLocalModul{f}{x^*}{\delta} = \rho(f(x^*+\delta),f(x^* + 0)).$$
		
	\item[(ii)] Although at a point of discontinuity\, $x^*$ at least one of the local moduli $\LeftLocalModul{f}{x^*}{\delta}$, $\RightLocalModul{f}{x^*}{\delta}$ does not tend to zero as $\delta $ tends to zero, for the local quasi-moduli we always have	
		$$
		\lim_{ \delta \to 0^+ } \NewLeftLocalModul{f}{x^*}{\delta}=0, \quad \lim_{ \delta \to 0^+ } \NewRightLocalModul{f}{x^*}{\delta} =0.
		$$ 
			
	\item[(iii)] The left local quasi-modulus of  $f$  at a point  $x^* \in (a,b]$ coincides with the left local modulus~\eqref{defin_LeftModul} of the function
$$
\widetilde{f}(x) = \begin{cases}
f(x), & x \ne x^*,\\
f(x^*-0), & x = x^*.
\end{cases}
$$
An analogous relation holds for the right local quasi-modulus. Clearly, at a point of continuity of $f$ the one sided local quasi-moduli and the one-sided local moduli of Section~\ref{Sect_Prelim_Regularity} coincide.
\end{enumerate}
\end{remark}

In the next two lemmas we derive results similar to those in Section~\ref{Sect_MetricSelections} for the local one-sided moduli.
\begin{lemma} \label{Lemma_NewLeftModulu_ChainFunct}
	Let $F \in \mathcal{F}[a,b]$, $x^* \in (a,b]$ and $c_{\chi,\phi}$ be a chain function corresponding to a partition $\chi$ and a~metric chain $\phi$. Then
	$$
	\NewLeftLocalModul{v_{c_{\chi,\phi}}}{x^*}{\delta} \le \NewLeftLocalModul{v_F}{x^*}{\delta+ |\chi|}, \quad \delta >0.
	$$
\end{lemma}
\begin{proof}
	We estimate $\, \NewLeftLocalModul{v_{c_{\chi,\phi}}}{x^*}{\delta} = v_{c_{\chi,\phi}}(x^*-0) - v_{c_{\chi,\phi}}(x^* - \delta)$.
	Let ${\chi = \{ a=x_0 < x_1 < \cdots < x_m=b\}}$. If ${x^* \not\in \chi}$, then $x^* \in (x_{k-1},x_{k})$ with some $1 \le k \le m$. If $x^* \in \chi$, then $x^* = x_k$ for some $1 \le k \le m$. In both cases $c_{\chi,\phi}(x) = c_{\chi,\phi}(x_{k-1})$ for $x_{k-1} \le x < x^*$, so that  $c_{\chi,\phi}(x^*-0) = c_{\chi,\phi}(x_{k-1})$. If $x_{k-1} \le x^* - \delta < x^*$, then $c_{\chi,\phi}(x^* - \delta) = c_{\chi,\phi}(x_{k-1})$ and $v_{c_{\chi,\phi}}(x^*-0) - v_{c_{\chi,\phi}}(x^* - \delta) = 0$. Otherwise there is $0 \le i < k-1$ such that $x_i \le x^* - \delta < x_{i+1}$ and $c_{\chi,\phi}(x^* - \delta) = c_{\chi,\phi}(x_{i})$. By the definitions of the metric chain and of the chain function we have 
	\begin{align*}
	& v_{c_{\chi,\phi}}(x^*-0) - v_{c_{\chi,\phi}}(x^* - \delta) =  \sum_{j=i}^{k-2} |c_{\chi,\phi}(x_{j+1}) - c_{\chi,\phi}(x_j)|
	\le \sum_{j=i}^{k-2}  \haus(F(x_{j+1}), F(x_j)) \\
	& \le V_{x_i}^{x_{k-1}}(F) = v_F(x_{k-1}) - v_F(x_i) \le v_F(x^*-0) - v_F(x^*-\delta-|\chi|) = \NewLeftLocalModul{v_F}{x^*}{\delta+ |\chi|}
	\end{align*}
	and we obtain the claim.
\end{proof}

\begin{lemma}\label{Lemma_NewLeftModulu_MetricSelection}
	Let $F \in \mathcal{F}[a,b]$,  $x^* \in (a,b]$  and $s \in \mathcal{S}(F)$. Then
	$$
	\NewLeftLocalModul{v_s}{x^*}{\delta} \le \NewLeftLocalModul{v_F}{x^*}{2\delta}, \quad \delta >0.
	$$
\end{lemma}
\begin{proof}
	Let $s \in \mathcal{S}(F)$ and $\delta > 0$. There exists a sequence of chain functions $\{c_n\}_{n \in \N}$ that corresponds to a sequence of partitions $\{\chi_n\}_{n \in \N}$ with   $|\chi_n| \to 0$ as $n \to \infty$ such that $s(x) = \lim_{n \to \infty}{c_n(x)}$, $x \in [a,b]$. Take $N \in \N$ so large that $|\chi_n| < \delta$ for all $n \ge N$. 
	
	We estimate $\NewLeftLocalModul{v_s}{x^*}{\delta} = v_s(x^* - 0) - v_s(x^* - \delta)$. Take $0 < t < \delta$. For each $n \ge N$ we have by Lemma~\ref{Lemma_NewLeftModulu_ChainFunct}
	$$
	V_{x^*-\delta}^{x^*-t}(c_n)=v_{c_n}(x^* - t) - v_{c_n}(x^* - \delta) \le \NewLeftLocalModul{v_{c_n}}{x^*}{\delta} \le \NewLeftLocalModul{v_F}{x^*}{\delta+ |\chi_n|} \le	\NewLeftLocalModul{v_F}{x^*}{2\delta}.
	$$
	By~Theorem \ref{theorem_LimitFunc_BV} we have
	$$
	v_{s}(x^* - t) - v_{s}(x^* - \delta)=V_{x^*-\delta}^{x^*-t}(s) \le \liminf\limits_{n \to \infty} V_{x^*-\delta}^{x^*-t}(c_n) \le \NewLeftLocalModul{v_F}{x^*}{2\delta}.
	$$
	Taking the limit as $t \to 0+$ we obtain the claim.
\end{proof}

Note that we cannot expect a bound for $\NewRightLocalModul{v_{c_{\chi,\phi}}}{x^*}{\delta}$ in terms of $\NewRightLocalModul{v_F}{x^*}{\delta+ \varepsilon}$. The reason is that in the definition of the chain function we use values on the left of a point $x^*$ that we cannot control by $\NewRightLocalModul{v_F}{x^*}{\delta}$. However, the following estimates hold true for a metric selection $s$.

\begin{lemma}\label{Lemma_NewRightModulu_MetricSelection}
	Let $F \in \mathcal{F}[a,b]$, $x^* \in [a,b)$ and $s \in \mathcal{S}(F)$. Then
	$$
	\NewRightLocalModul{v_s}{x^*}{\delta} \le \NewRightLocalModul{v_F}{x^*}{\delta}, \quad \delta >0.
	$$
\end{lemma}
\begin{proof}	
	Let $s \in \mathcal{S}(F)$ and $\delta > 0$. Let $\{c_n\}_{n \in \N}$ be a sequence of chain functions like in the proof of Lemma~\ref{Lemma_NewLeftModulu_MetricSelection}.
	We estimate $ \NewRightLocalModul{v_s}{x^*}{\delta} = v_{s}(x^* + \delta) - v_{s}(x^* + 0)$.
	Take $0 < t <\delta$.  There is $N \in \N$ such that $|\chi_n| < t$ for all $n \ge N$. Then the interval $(x^*, x^* + t)$ contains at least one point of the partition $\chi_n$, $n \ge N$.  Let $\chi_n = \{ a=x_0^n < x_1^n < \cdots < x_{m(n)}^n=b\}$. There is $0 \le k(n) \le m(n)-1$ such that $x^* + t \in [x_{k(n)}^n,x_{k(n)+1}^n)$.  It holds $x_{k(n)}^n > x^*$. 
	
	If $x^* + \delta \in [x_{k(n)}^n,x_{k(n)+1}^n)$, then $c_n(x^* + t) = c_n(x^* + \delta) = c_n(x_{k(n)}^n)$, so that $v_{c_n}(x^* + \delta) - v_{c_n}(x^* + t) = 0$. Otherwise there is $k(n) < i(n) \le m(n)-1$ such that $x^* + \delta \in [x_{i(n)}^n, x_{i(n)+1}^n)$, or $x^* + \delta = b = x_{m(n)}^n$ so that $i(n) = m(n)$. In both cases $c_n(x^* + \delta) = c_n(x_{i(n)}^n)$. Therefore,
	\begin{align*}
	V_{x^*+t}^{x^*+\delta}(c_n)&= v_{c_n}(x^* + \delta) - v_{c_n}(x^* + t) =  \sum_{j=k(n)}^{i(n)-1} |c_n(x_{j+1}^n) - c_n(x_j^n)|
	\le \sum_{j=k(n)}^{i(n)-1}  \haus(F(x_{j+1}^n), F(x_j^n)) \\
	& \le V_{x_{k(n)}}^{x_{i(n)}}(F) = v_F(x_{i(n)}) - v_F(x_{k(n)}) \le v_F(x^* + \delta) - v_F(x^*+0)  = 	\NewRightLocalModul{v_F}{x^*}{\delta}
	\end{align*}
	for each $n \ge N$. By~Theorem \ref{theorem_LimitFunc_BV} we have
	$$
	v_{s}(x^* + \delta) - v_{s}(x^* + t)=V_{x^*+t}^{x^*+\delta}(s) \le \liminf\limits_{n \to \infty} V_{x^*+t}^{x^*+\delta}(c_n) \le \NewRightLocalModul{v_F}{x^*}{\delta}.
	$$
	Taking the limit as $t \to 0+$ we obtain the claim.	
\end{proof}

In the next definition we introduce several classes of periodic vector-valued functions. 

\begin{defin}\label{Def_EqvivClass}
		Given $B > 0$, a point $x \in \R$, a closed interval $I \subset \R$ and a modulus-bounding function $\omega$, we define the following classes of functions.
	\begin{enumerate}
		\item[(i)] $\EqvClass{d}{B}{x}{\omega}$  is the class of all $2\pi$-periodic functions  $f : \RR \to \RR^d$ satisfying
		$$		
		V_{-\pi}^{\pi}(f) \le B \quad  \mbox{and} \quad 
		\NewLeftLocalModul{v_f}{x}{\delta} \le \omega(\delta), \quad \NewRightLocalModul{v_f}{x}{\delta} \le \omega(\delta)
		$$
		for all $0< \delta  \le \pi$.
		
		\item[(ii)]  $\EqvClass{d}{B}{I}{\omega} = {\displaystyle \bigcap_{z\in I}\EqvClass{d}{B}{z}{\omega}} $.
		
		\item[(iii)]  $\mathscr{C}\EqvClass{d}{B}{I}{\omega} = \EqvClass{d}{B}{I}{\omega} \cap \cal{C}_d(I) $, where $\cal{C}_d(I)$ is the class of functions $f : \RR \to \RR^d$ which are continuous on $I$.
	\end{enumerate}
\end{defin}

\begin{remark} \label{Remark_Class_F_Coordinates_Jump}
	It is easy to conclude from the equivalence of norms on $\Rd$ that if $f : \RR \to \RR^d$, $f = \begin{pmatrix} f_1 \\ \vdots \\ f_d \end{pmatrix}$, and $f \in \EqvClass{d}{B}{x}{\omega}$, then $f_j \in \EqvClass{1}{KB}{x}{K\omega}$, $j = 1, \ldots, d$, with a constant $K > 0$ depending only on the underlying norm on $\Rd$.
\end{remark}
In view of Remark~\ref{Remark_Class_F_Coordinates_Jump} we formulate the subsequent results only for functions $f : \RR \to \RR$.

The theorem below is an extension of the Dirichlet-Jordan Theorem for the class  $\EqvClass{1}{B}{x}{\omega}$.
To establish the result, we carefully go through the proof of the Dirchlet-Jordan Theorem in~\cite[Chapter  II]{Zygmund:TrigSeries} and examine the estimates. The proof is given in Appendix B. 

\begin{theorem}\label{Theo_FourierConv_Class_Jump}
	Let $B > 0$, $x \in \R$ and $\omega$ be a modulus-bounding  function. Then for all $f \in \EqvClass{1}{B}{x}{\omega}$ and each  $\delta\in(0,\pi]$  we have
	\begin{equation} \label{Fourier_Class_Est}
	\left|  \mathscr{S}_n f(x) - \frac{1}{2}\big ( f(x+0) + f(x-0) \big ) \right| \le \frac{2B}{\pi n} \left( 1+6\cot \left(\frac{\delta}{2} \right) \right)+8C\omega(\delta), \quad n\in\N,
	\end{equation}
	where $C$ is the constant from Lemma~\ref{Lemma:Int_D_Estimate}.
\end{theorem}
\noindent In view of Remark~\ref{Remark_Value_C} one can take $C=2$ in~\eqref{Fourier_Class_Est}. 
\medskip 

The next corollary follows from the above theorem.

\begin{corol}\label{lemma_Convergence}
Let $B > 0$, $x \in \R$ and $\omega$ be a modulus-bounding  function satisfying  $\lim_{ \delta \to 0^+ }\omega(\delta)=0$. Then
	$$
	\lim_{ n \to \infty }\, \sup\left\{ \left|  \mathscr{S}_n f(x) - \frac{1}{2}\big(f(x+0) + f(x-0)\big) \right| \ : \ f\in\EqvClass{1}{B}{x}{\omega} \right\} =0.
	$$
\end{corol}
\begin{proof}
	Take an arbitrary $\eps>0$. Fix $\delta>0$ such that $\omega(\delta)<\frac{\eps}{16C}$. Choose $n$ large enough such that ${\frac{2B}{\pi n} \left( 1+6\cot \left(\frac{\delta}{2} \right) \right) < \frac{\varepsilon}{2} }$. Then by~\eqref{Fourier_Class_Est} we have
	$$
	\left|  \mathscr{S}_n f(x) - \frac{1}{2}\big(f(x+0) + f(x-0)\big) \right| < \frac{\eps}{2} + \frac{\eps}{2} = \varepsilon 
	$$
	for all $f \in \EqvClass{1}{B}{x}{\omega}$, and the statement follows.
\end{proof}

For $ f \in \EqvClass{1}{B}{I}{\omega}$ the estimate in the right-hand side of~\eqref{Fourier_Class_Est}  does not depend on $x \in I$. We arrive at the following statement.

\begin{corol}\label{Corol_UniformConv_BV}
	Let $B > 0$, $I \subset \R$ be a closed interval and $\omega$ be a modulus-bounding  function satisfying  ${\lim_{\delta \to 0+}{\omega(\delta)} = 0}$. Then
	$$
	\lim_{ n \to \infty }\, \sup\left\{ \left|  \mathscr{S}_n f(x) - \frac{1}{2}\big(f(x+0) + f(x-0)\big ) \right| \ : \ x\in I, \ f\in\EqvClass{1}{B}{I}{\omega} \right\} =0.
	$$
\end{corol}

Finally, if $f$ is in addition continuous in $I$ then the Fourier series of $f$ converges to $f$ on $I$, and the statement above takes the following form.

\begin{corol}\label{Corol_UniformConv_CBV}
	Under the assumptions of Corollary~\ref{Corol_UniformConv_BV} we have
	$$
	\lim_{ n \to \infty }\, \sup\left\{ \left|  \mathscr{S}_n f(x) - f(x) \right| \ : \  x \in I, \; f\in \mathscr{C}\EqvClass{1}{B}{I}{\omega} \right\} =0.
	$$
\end{corol}

\medskip


\subsection{Extension to SVFs}\label{Subsec_FourierSVFs}
We define the Fourier series of set-valued functions via the integral representation~\eqref{Dirichlet-repr} using the weighted metric integral.

\begin{defin}\label{Defin_FourierSeries}
	Let $\Map{[-\pi,\pi]}$. The \textbf{metric Fourier series} of $F$ is the sequence of the set-valued functions $\{\mathscr{S}_nF\}_{n \in \N}$, where $\mathscr{S}_nF $ is a SVF defined by
	$$
	\mathscr{S}_nF (x) = \frac{1}{\pi} {\scriptstyle(\cal M_{\partial_{n,x}})} \int_{-\pi}^{\pi} \partial_{n,x}(t)  F(t) dt, \quad x \in [-\pi,\pi], \quad n \in \N,
	$$
	whenever the integrals above exist.
\end{defin}

For $F \in \mathcal{F}[-\pi,\pi]$ the integrals in Definition~\ref{Defin_FourierSeries} exist. Moreover, each $\partial_{n,x} = D_n(x - \cdot)$ for  fixed $n \in \N$ and ${x \in \RR}$ is of bounded variation on each finite interval. Hence, if $F \in  \mathcal{F}[-\pi,\pi]$, then  the set-valued functions ${\mathscr{S}_nF}$  have compact images by Proposition~\ref{WeihtedMetInt_is_compact}.  By Theorem~\ref{Theo_W-MetrInt=W-IntOfMetrSel} we have
\begin{equation} \label{Fourier_Repr_Metric_Sel}
\mathscr{S}_nF (x) = \left \{ \mathscr{S}_n s (x) \ : \ s \in\mathcal{S}(F) \right \} = \left\{ \frac{1}{\pi} \int_{-\pi}^{\pi} D_n(x-t) s(t) dt \ : \ s \in\mathcal{S}(F) \right\}, \quad x \in [-\pi,\pi].
\end{equation}

Note that we do not expect  metric selections $s$ in this definition to be periodic. In fact, even if the set-valued function $F$ itself is periodic, it can have metric selections that are not periodic (see Figure~\ref{Figure_Selections}).

\begin{center}
	\includegraphics[width=3.5 in]{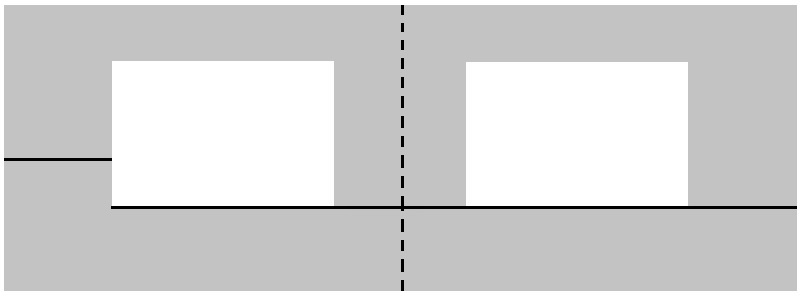}\\
	\begin{footnotesize}
	\end{footnotesize}
\end{center}
\begin{pict}\label{Figure_Selections}
	\footnotesize
	{\centering {An example of a non-periodic metric selection of a periodic SVF }\\}
\end{pict}

\smallskip
For $F \in \mathcal{F}[-\pi,\pi]$ and $x \in (-\pi,\pi)$ we define
\begin{equation}\label{LimitSet-Fourier}
A_F(x) = \left\{ \frac{1}{2} \left( s(x+0) + s(x-0) \right) \ : \ s \in \mathcal{S}(F) \right\}.
\end{equation}
We show that this is the limit set of the Fourier approximants.

\begin{propos} \label{Propos_Est_SVF}
Let $F \in  \mathcal{F}[-\pi,\pi]$ and $x \in (-\pi,\pi)$. Then   there exists $\delta_0=\delta_0(x) > 0$ such that for all $\delta \in (0,\delta_0] $ and $n \in \N$ the following estimate holds
\begin{equation} \label{Fourier_Conv_SVF_Jump_Estimate}
 \haus \left( \mathscr{S}_nF(x), A_F(x) \right) \le
K \left[ \frac{ V_{-\pi}^{\pi} (F)}{n} \left( 1+6\cot \left(\frac{\delta}{2} \right) \right) + \omega(\delta) \right],
\end{equation}
where $\omega(\delta)=\max\big \{\NewLeftLocalModul{v_F}{x}{2\delta} , \NewRightLocalModul{v_F}{x}{\delta} \big \}$ and $K > 0$ is a constant that depends only on the underlying norm in the space $\R^d$. 
 \end{propos}

\begin{proof}
First we observe that $\omega(\delta)$ is a modulus-bounding function by its definition. 

Next, by~\eqref{Fourier_Repr_Metric_Sel}, \eqref{LimitSet-Fourier} we have 
		\begin{equation} \label{Haus_Le_Sup_Sel_Jump}
		\haus \left( \mathscr{S}_nF(x),A_ F(x) \right) \le \sup{ \left\{ \left|\mathscr{S}_ns(x) - \frac{1}{2} \left ( s(x+0) + s(x-0) \right ) \right| \ : \ s \in \mathcal{S}(F) \right\} }. 
		\end{equation}
Indeed, for any $y \in \mathscr{S}_nF(x)$ and for any selection $s \in \mathcal{S}(F)$ with $y=\mathscr{S}_ns(x)$, the following holds: \linebreak
$
\left |y - \frac{1}{2} \left ( s(x+0) + s(x-0) \right ) \right | \ge \dist (y, A_ F(x)).
$	
Similarly, for any $z \in  A_ F(x)$ and for any $s \in \mathcal{S}(F)$ such that $z= \frac{1}{2} \left ( s(x+0) + s(x-0) \right ) $ we have $\left |z - \mathscr{S}_ns(x) \right| \ge \dist (z, \mathscr{S}_nF(x) )$. The last two inequalities, in view of~\eqref{defin_Haus}, imply~\eqref{Haus_Le_Sup_Sel_Jump}.

	Let $s \in \mathcal{S}(F)$, and let $\tilde{s}$ be the $2\pi$-periodic function that coincides with $s$ on $[-\pi,\pi)$. Clearly, $\mathscr{S}_n s = \mathscr{S}_n \tilde{s}$.  By Result~\ref{Result_MetrSel_InheritVariation} we have $V_{-\pi}^{\pi}( \tilde{s} ) \le 2 V_{-\pi}^{\pi} (F)$; the factor $2$ here comes because of a possible jump at the point $\pi$. 
Since $x$ lies in the open interval $(-\pi,\pi)$, there exists $\delta_0 > 0$ such that $[x- \delta_0, x + \delta_0] \subset (-\pi,\pi)$ and therefore $\tilde{s}$ coincides with $s$ in the interval $[x- \delta_0, x + \delta_0]$. Thus by  Lemmas~\ref{Lemma_NewLeftModulu_MetricSelection} and \ref{Lemma_NewRightModulu_MetricSelection}
	$$
	\NewLeftLocalModul{v_{\tilde{s}}}{x}{\delta} \le \NewLeftLocalModul{v_F}{x}{2\delta} \le \omega(\delta), \quad
	\NewRightLocalModul{v_{\tilde{s}}}{x}{\delta} \le \NewRightLocalModul{v_F}{x}{\delta} \le \omega(\delta), \quad 
	\delta \in(0,\delta_0].
	$$
For $\delta > \delta_0$, we redefine $\omega(\delta)$ in a non-decreasing way  so that the estimates  $\NewLeftLocalModul{v_{\tilde{s}}}{x}{\delta} \le \omega(\delta)$, $\NewRightLocalModul{v_{\tilde{s}}}{x}{\delta} \le \omega(\delta)$ hold for all $\delta \in (0,\pi]$.  We achieve it  by putting  $\omega(\delta) = 2 V_{-\pi}^{\pi}(F)$ for $\delta_0 < \delta \le \pi$.

By Remark~\ref{Remark_Class_F_Coordinates_Jump}, there exist a constant $K_1 > 0$ such that  for each metric selection $s \in \mathcal{S}(F)$, each coordinate of its $2\pi$-periodization $\tilde{s}_j$, $j = 1, \ldots, d$, lies in the class $\EqvClass{1}{2 K_1 V_{-\pi}^{\pi} (F)}{x}{ K_1\omega}$. 
Applying  Theorem~\ref{Theo_FourierConv_Class_Jump} to all $\tilde{s}_j$, $j = 1, \ldots, d$, we obtain  for each $s\in \mathcal{S}(F)$
\begin{align*}\label{formula_1}
& \left| \mathscr{S}_ns(x) -  \frac{1}{2}(s(x+0) + s(x-0)) \right|
\le K_2 \max_{j = 1, \ldots,d} { \left| \mathscr{S}_ns_j(x) -  \frac{1}{2}(s_j(x+0) + s_j(x-0)) \right| } \\
& \le K_2 \left[ \frac{2K_1 V_{-\pi}^{\pi}(F)}{\pi n} \left( 1+6\cot \left(\frac{\delta}{2} \right) \right) + 8C K_1 \omega(\delta) \right],
\end{align*}
where the constant $K_2 > 0$ depends only on the underlying norm in $\R^d$. In view of~\eqref{Haus_Le_Sup_Sel_Jump} the claim follows with $K= 2K_1 K_2 \max\{\frac{1}{\pi}, 4C\}$, where $C$ is defined in~\eqref{Int_D_Estimate}.

\end{proof}

The next two theorems are the main results of the paper.

\begin{theorem}\label{Theo_FourierConv_Main_Jump}
Let $F \in  \mathcal{F}[-\pi,\pi]$ and $x \in (-\pi,\pi)$. Then 	
	\begin{equation} \label{Fourier_Conv_SVF_Jump}
	\lim_{n \to \infty}{ \haus \left( \mathscr{S}_nF(x), A_F(x) \right) } = 0.
	\end{equation}
\end{theorem}
\begin{proof}
Let $\omega(\delta)=\max\big \{\NewLeftLocalModul{v_F}{x}{2\delta} , \NewRightLocalModul{v_F}{x}{\delta} \big \}$.
By Remark~\ref{Remark_NewModuli}(ii), $\omega(\delta) \to 0$ as $\delta \to 0+$. To prove~\eqref{Fourier_Conv_SVF_Jump}, take an arbitrary $\varepsilon > 0$ and choose in~\eqref{Fourier_Conv_SVF_Jump_Estimate}  first $\delta \in (0, \delta_0(x)]$ so small that $K \omega(\delta) < \frac{\varepsilon}{2}$. Then by choosing $n$ so large that $K  V_{-\pi}^{\pi} (F) \frac{1}{n} \left( 1+6\cot \left(\frac{\delta}{2} \right) \right)  < \frac{\varepsilon}{2}$ we complete the proof. 
\end{proof}
}

In case $F$ is continuous, its Fourier series converges to $F$ in the Hausdorff metric. Namely, the following holds true.
\begin{theorem}\label{Theo_FourierConv_Main}
	Let $F \in  \mathcal{F}[-\pi,\pi]$ and let $F$ be continuous at $x \in (-\pi,\pi)$. Then
	$$
	\lim_{n \to \infty}{ \haus \left( \mathscr{S}_nF(x), F(x) \right) } = 0.
	$$
	If $F$ is continuous in a closed interval $I \subset (-\pi,\pi)$, then the convergence is uniform in $I$. 
\end{theorem}
\begin{proof}
The first statement of the above theorem is an immediate consequence of Theorem~\ref{Theo_FourierConv_Main_Jump}. For the second statement note that there exists $\delta_0 > 0$ such that $[x-\delta_0,x+\delta_0] \subset (-\pi,\pi)$ for all $x \in I$. Defining $\omega(\delta)$ as in the proof of Proposition~\ref{Propos_Est_SVF} and applying Corollary~\ref{Corol_UniformConv_BV}, we obtain the result.
\end{proof}


 \section{On the limit set of the Fourier~approximants}

In the previous section we proved that the sequence $\{ \mathscr{S}_nF(x)\}_{n \in \N} $ converges  at a point $x$ where $F$ is discontinuous to the set $A_F(x) = \left \{ \frac{1}{2} \left( s(x+0) + s(x-0) \right) \ : \ s \in \mathcal{S}(F) \right \} $. An interesting question is to describe the set $A_F(x)$ in terms of the values of $F$. At the moment we do not have a satisfactory answer to this question. 

The two statements below give some idea about the structure of a set-valued function $F$ and its metric selections at a point $x$  where $F$ is discontinuous.
 
 \begin{propos} \label{Propos_LimitsInGraph}
 	For $F\in \calF$ and $x \in (a,b)$  we have $F(x-0) \cup F(x+0) \subseteq F(x)$. 
 \end{propos}
\begin{proof}
We show that $F(x-0) \subseteq F(x)$, the proof for $F(x+0)$ is similar. 

Since $F$ is bounded, we can restrict our consideration to a bounded region of $\R^d$, so that the convergence in the Hausdorff metric is equivalent to the convergence in the sense of Kuratowski (see Remark~\ref{Remark_Kurat=Haus}).

Consider $y \in F(x-0)$. Take an arbitrary sequence $\{x_n\}_{n \in \N}$ with $x_n < x$, $n \in \N$, and $x_n \to x$, $n \to \infty$. Since $F(x-0)$ coincides with the lower Kuratowski limit $\liminf_{t \to x-0}{F(t)}$, for each $n$ there exists $y_n \in F(x_n)$ such that $y_n \to y$, $n \to \infty$.  We have $(x_n,y_n) \in \Graph{(F)}$ for each $n \in \N$ and $(x_n,y_n) \to (x,y)$, $n \to \infty$. Since $\Graph{(F)}$ is  closed, it follows that $(x,y) \in \Graph{(F)}$, and thus $y \in F(x)$. This implies that $F(x-0) \subseteq F(x)$.
\end{proof}

\begin{propos}	
 	For $F \in \mathcal{F}[a,b]$ 
 	$$
 	F(x-0) = \{ s(x-0) \ : \ s \in \mathcal{S}(F) \}, \quad x \in (a,b],  \quad \text{and} \quad
 	F(x+0) = \{ s(x+0) \ : \ s \in \mathcal{S}(F) \}, \quad x \in [a,b).
 	$$
 \end{propos}
 \begin{proof}
 	We prove the first claim, the proof of the second one is similar. 
 	
 	Fix $x\in(a,b]$. The inclusion $\{ s(x-0) \ :\ s\in \mathcal{S}(F)\} \subseteq F(x-0)$ follows from the fact that $F(x-0)$ coincides with the Kuratovski upper limit $\limsup_{t \to x-0}{F(t)}$ (see Remark~\ref{Remark_Kurat=Haus}). It remains to show  $F(x-0) \subseteq \{s(x-0) \ :\ s\in \mathcal{S}(F)\} $.
 	Define a multifunction $\widetilde F: [a,b] \to \Comp$ by
 	$$
 	\widetilde F(t)=
 	\left \{ \begin{array}{ll}
 	F(t), & t \neq x, \\
 	F(x-0), & t=x.
 	\end{array}
 	\right.
 	$$
 	Clearly,  $\widetilde F$ is left continuous at~$x$ and $\widetilde F \in \calF$.
 	By Result~\ref{Result_MetSel_ThroughAnyPoint_Repres} $\widetilde F$ has a representation by its metric selection. By Proposition~\ref{Propos_LimitsInGraph} $\widetilde F(x) \subseteq F(x)$, and thus $\mathcal{S}(\widetilde F)  \subseteq \mathcal{S}(F)$.
 	
  Now, let $y \in F(x-0) = \widetilde F(x) \subseteq F(x)$.  There exists a selection $s\in \mathcal{S}(\widetilde F) \subseteq \mathcal{S}( F)$ such that $y=s(x)$.  Since $\widetilde F$ is left continuous at~$x$, by Theorem~\ref{theorem_LocLeftModuli_s<=LocLeftModuli_v_f} $s$ is also left continuous at~$x$. Thus, ${y = s(x) = s(x-0) \in \{s(x-0) \ :\ s\in \mathcal{S}(F)\} }$. 
 \end{proof}

 In view of the last proposition and by the definition of $A_F(x)$ (see~\eqref{LimitSet-Fourier}), we conclude
 $$
 A_F(x) \subseteq \frac{1}{2} F(x-0) + \frac{1}{2} F(x+0),
 $$
 where the right-hand side is the Minkowski average which might be much larger than $A_F(x)$. 
 
One could conjecture that $A_F(x)$ coincides with the metric average of $F(x-0)$ and $F(x+0)$,  namely
$$
A_F(x)=\frac{1}{2} F(x-0) \oplus \frac{1}{2} F(x+0), 
$$
where 
$$
\frac{1}{2} F(x-0) \oplus \frac{1}{2} F(x+0)=\left \{ \frac{1}{2} y^- + \frac{1}{2} y^+ \, : \, (y^-,y^+) \in \Pair{F(x-0)}{F(x+0)} \right \}.
$$
It is easy to see that a sufficient condition for the inclusion
\begin{equation}\label{Formula_InclusionConjecture_1}
A_F(x) \subseteq \frac{1}{2} F(x-0) \oplus \frac{1}{2} F(x+0)
\end{equation}
is  the property
\begin{equation}\label{formula_4}
\big ( s(x-0),s(x+0) \big ) \in \Pair{F(x-0)}{F(x+0)}
\end{equation} 
for any $s\in \mathcal{S}(F)$. However, \eqref{formula_4} is not always true.
The next example provides a counterexample to both~\eqref{formula_4}  and~\eqref{Formula_InclusionConjecture_1}.

\begin{example} \label{Bad_example_with_balls}
Let $B(x_1,x_2)$ denote the closed disc of radius $1$ with center at the point $(x_1,x_2)$, and let $x\in (-\pi,\pi)$.
Consider the function $F: [-\pi,\pi] \to {\mathrm{K}(\R^2)}$, $F \in \mathcal{F}[-\pi,\pi]$, defined by 
 	$$
 	F(t) = \begin{cases}
 	B(-2,2), & t \in [-\pi,x), \\
 	B(-2,2) \cup \{(0,0)\}  \cup B(2,2), & t=x, \\
 	B(2,2), & t \in (x,\pi],
 	\end{cases}
 	$$
 	and its metric selection 
 	$$
 	s(t) = \begin{cases}
 	(-2 + \frac{\sqrt{2}}{2}, 2 - \frac{\sqrt{2}}{2}), & t \in [-\pi,x),  \\
 	(0,0), & t=x,\\
 	(2 - \frac{\sqrt{2}}{2}, 2 - \frac{\sqrt{2}}{2}), & t \in (x,\pi].
 	\end{cases}
 	$$
 	First we show that~\eqref{formula_4} does not hold. It is easy to see that $s(x-0) = (-2 + \frac{\sqrt{2}}{2}, 2 - \frac{\sqrt{2}}{2}) = \Pi_{F(x-0)}((0,0))$ is the projection
  of ${(0,0) \in F(x)}$ on ${F(x-0)}$, and ${s(x+0) = (2 - \frac{\sqrt{2}}{2}, 2 - \frac{\sqrt{2}}{2}) = \Pi_{F(x+0)}((0,0))}$ is the projection of ${(0,0)}$ on ${F(x+0)}$.  
 	On the other hand, the pair $\big ( s(x-0) , s(x+0) \big )$ is not a metric pair of $(F(x-0),F(x+0))$  since the line connecting the points $s(x-0)$ and $s(x+0)$ does not pass through any of the centers of the two discs. 		
 	By similar geometric arguments one can show that ${\frac{1}{2} ( s(x-0) + s(x + 0)) = (0, 2 - \frac{\sqrt{2}}{2}) \in A_F(x)}$, but does not belong to ${\frac{1}{2} F(x-0) \oplus \frac{1}{2} F(x+0)}$.
 	
Note that in this example $F(x-0)\cup F(x+0) \neq F(x)$, and that the selection $s$ for which~\eqref{formula_4} does not hold satisfies $s(x) \notin F(x-0)\cup F(x+0)$.
 \end{example}

Also the reverse inclusion to~\eqref{Formula_InclusionConjecture_1}, $A_F(x) \supseteq \frac{1}{2} F(x-0) \oplus \frac{1}{2} F(x+0)$, does not hold in general. The next example demonstrates this.

\begin{example} \label{Bad_example_with_lines}

Consider the set-valued function $F: [-\pi,\pi] \to \mathrm{K}(\R)$ defined by
$$
F(t) = \begin{cases}
\left\{ -\frac{1}{4}, 0, \frac{1}{4} \right\}, & t \in [-\pi,x),\\
\left\{ -1, -\frac{1}{4}, 0, \frac{1}{4}, 1 \right\}, & t = x, \\
\left\{ -1 + t - x, 1 + t - x \right\}, & t \in (x,\pi],
\end{cases}
$$
where $x \in (-\pi,\pi)$.
We have $F(x-0) = \left\{ -\frac{1}{4}, 0, \frac{1}{4} \right\}$, $F(x+0) = \{-1, 1 \}$, and their metric average is 
$\frac{1}{2} F(x-0) \oplus \frac{1}{2} F(x+0) = \left\{ -\frac{5}{8}, -\frac{1}{2}, \frac{1}{2}, \frac{5}{8} \right\}$. We show that $\frac{1}{2} \in \frac{1}{2} F(x-0) \oplus \frac{1}{2} F(x+0)$ does not belong to $A_F(x)$, i.e., there is no metric selection $s$ of $F$ such that 
\begin{equation} \label{example_sel_prop}
\frac{1}{2} = \frac{1}{2}(s(x-0) + s(x+0)).
\end{equation}
Indeed, if~\eqref{example_sel_prop} is fulfilled for a selection  $\hat{s}$  of $F$, then for this selection we necessarily have $\hat{s}(t) = 0$ for $t \in [-\pi,x)$ and $\hat{s}(t) = 1 + t - x$ for $t \in (x,\pi]$ (with an arbitrary choice of the value $s(x) \in F(x)$). But such $\hat{s}$ cannot be a metric selection, because there are no chain functions that would  lead to such a selection. 
 The only chain functions which might converge to $\hat{s}$ are constant with the value $0$ on the left of $x$ and piecewise constant functions with values sampled from $1 + t -x$ on the right of $x$, possibly except for the interval between two neighboring points of the partition that contains the point $x$.

But no chain function can take the value $0$ on the left of $x$ and the value ${1 + t -x}$ on the right  of $x$. Indeed, if $x$ is not a point of the partition, then this is impossible because the closest point to $0$ in the set ${F(t) = \left\{ -1 + t - x\, ,\, 1 + t - x \right\}}$, $t > x$, is $ -1 + t - x$ and not $ 1 + t - x$, and the closest point to $1 + t -x$ in the set $F(t) = \left\{ -\frac{1}{4}, 0, \frac{1}{4} \right\}$, $t < x$, is $\frac{1}{4}$ and not $0$. If $x$ is a point of the partition, then  the value of a chain function at $x$ is one of the five values from ${F(x) = \left\{ -1, -\frac{1}{4}, 0, \frac{1}{4}, 1 \right\}}$. The choices  $0$ and~$1$ are impossible because of the reasons explained above. But also the other three choices are impossible, since the pointwise limit of the chain functions would not be equal to $0$ on the left of~$x$.

\end{example}
 
Yet, the conjecture $A_F(x)=\frac{1}{2} F(x-0) \oplus \frac{1}{2} F(x+0)$ or a weaker form of it might be true for functions $F$ from a certain subclass of $\calF$.

\bigskip

{\bf Acknowledgement.}
A considerable part of this work was done during a two weeks long stay of the authors at the Mathematisches Forschungsinstitut Oberwolfach in the frames of the ``Research in Pairs'' program.  We thank the Mathematisches Forschungsinstitut Oberwolfach for the hospitality, excellent working conditions and the inspiring atmosphere.


\medskip

{\large \bf{Appendix A: Proof of Theorem~\ref{theorem_PW-limit_of_MS_isMS}}\label{Section_Appendix_A}}
\medskip

\noindent \textbf{Theorem~\ref{theorem_PW-limit_of_MS_isMS}.}
\textit{
	For $F\in \calF$, the pointwise limit of a sequence of metric selections of~$F$ is a metric selection of $F$.
}

\begin{proof}
	Let $s^{\infty}$ be the pointwise limit of a sequence $\{s_n\}_{n \in \N}$ of metric selections of~$F$.
	Since $F(x)$ is closed for each $x \in [a,b]$,\,  $s^{\infty}$ is a selection of $F$. By Result~\ref{Result_MetrSel_InheritVariation} we have  $V_a^b(s_n) \le V_a^b(F)$ for each~$n \in \N$. Theorem~\ref{theorem_LimitFunc_BV} implies $V_a^b(s^{\infty}) \le V_a^b(F)$.
	
	For each $s_n$, there is a sequence of partitions $\{ \chi_{n,k} \}_{k\in \N}$ with $\displaystyle \lim_{k \to \infty} |\chi_{n,k}| =0$ and a sequence of corresponding chain functions $\{c_{n,k}\}_{k \in \N}$ with $ \lim_{k \to \infty} c_{n,k}(x) =s_n(x)$ pointwisely for $x \in [a,b]$.
	
	Without loss of generality we may assume that each of the sequences $\{ \chi_{n,k} \}_{k\in \N}$  satisfies the property $ |\chi_{n,k}|< \frac{1}{2^n} $ for all $k \ge n$. Indeed, for fixed $n$ we have $|\chi_{n,k}| \to 0$ as $k  \to \infty$, and thus there exists $K_n \in \N$ such that $|\chi_{n,k}|<\frac{1}{2^n}$ for all $k \ge K_n$.
	If $K_n \le n$,  then our assumption already holds. If $K_n > n$, we remove ${\chi_{n,1},\ldots,\chi_{n,K_n-n}}$ from the sequence  $\{ \chi_{n,k} \}_{k \in \N}$.
	
	Denote by $D$ the set consisting of all points of all partitions $\{ \chi_{n,k} \}_{n,k \in \N}$, all points of discontinuity of the functions $ \{s_n\}_{n\in \N}$,  all points of discontinuity of  $s^\infty$ and  all points of discontinuity of $F$. The set $D$ is dense in $[a,b]$.  Since the set of discontinuities of each BV function is at most countable, the set $D$ is countable. We  order it as a sequence $D=\{x_j\}_{j \in \N}$.
	
	To show that $s^\infty$ is a metric selection, we will construct a sequence of chain functions that converges pointwisely to $s^\infty$ for all $x\in [a,b]$.
	
	For each $n$, there is an index $k_n \ge n$ such that 
	\begin{equation}\label{formula1}
	|s_n(x_j)-c_{n,k_n}(x_j)| < \frac{1}{2^n}, \quad j=1, \ldots,n. 
	\end{equation}
	Since $k_n \ge n$, we also have
	$$
	|\chi_{n,k_n}| < \frac{1}{2^n}.
	$$
	For simplicity, we denote $\psi_n=c_{n,k_n}$ and $\chi_n= \chi_{n,{k_n}}$. Clearly, $|\chi_n| \to 0$ as $n \to \infty$. We will show that there is a subsequence of $\{\psi_n\}_{n \in \N}$ that converges  to $s^\infty$ pointwisely on $[a,b]$. 
	
	First we show that $\{ \psi_n \}_{n\in \N}$ converges to $s^\infty$  on the set~$D$. Fix $x \in D$.  Then $x=x_{j^*}$ for some $j^* = j^*(x) \in\N$. Given $\varepsilon >0$, choose $N(x,\varepsilon) \in \N$ such that ${|s^\infty(x) - s_n(x)| < \frac{\varepsilon}{2}}$ for all $n>N(x,\varepsilon)$. By \eqref{formula1}, $|s_n(x)-\psi_n(x)| < \frac{1}{2^n}$ for all~$n \ge j^*(x)$. For $n> \log_2(1/\varepsilon)+1$ we have $\frac{1}{2^n}<\frac{\varepsilon}{2}$.  Thus for every $n > \max\{ N(x,\varepsilon),  j^*(x), \log_2(1/\varepsilon)+1 \}$ we have
	$$
	|s^\infty(x) - \psi_n(x)| \le |s^\infty(x) - s_n(x)| + |s_n(x) - \psi_n(x)| < \frac{\varepsilon}{2}+\frac{\varepsilon}{2}=\varepsilon,
	$$
	which proves that $\lim_{n \to \infty}\psi_n(x) = s^\infty(x)$ for $x \in D$.
	
	By Result~\ref{Result_ChainFunct_Bounded&BV} the functions $\psi_n$, $n \in \N$, are uniformly bounded and of uniformly bounded variation.  
	By  Helly's Selection Principle applied consequently to each component of $\psi_n : [a,b] \to \Rd$, there exists a subsequence that converges pointwisely to a certain function $\psi^\infty: [a,b] \to \Rd$. For simplicity, we will denote this subsequence again by $\{\psi_n \}_{n \in \N}$.  Clearly, $\psi^\infty(x) =s^\infty(x)$ for all $x\in D$.
	
	It remains to show that ${\psi^\infty(x) =s^\infty(x)}$ for all $x \in [a,b]\setminus D$. Fix  $x \in [a,b]\setminus D$.  For an arbitrary $r \in D$ and each $n \in \N$ we have
	\begin{equation}\label{formula2}
	|\psi^\infty(x)-s^\infty(x)| \le |\psi^\infty(x)-\psi_n(x)|+|\psi_n(x)-\psi_n(r)|+|\psi_n(r)-s^\infty(r)|+|s^\infty(r)-s^\infty(x)|.
	\end{equation}
	Take $\varepsilon>0$. 
	Since $\lim_{n \to \infty} \psi_n(x) = \psi^\infty(x)$, there exists $N_1(x,\varepsilon) \in \N$ such that 
	$$
	|\psi^\infty(x)-\psi_n(x)|<\frac{\varepsilon}{4}, \quad   n>N_1(x,\varepsilon). 
	$$
	Also for each $r \in D$ we have $\lim_{n \to \infty}  \psi_n(r) = \psi^\infty(r)= s^\infty(r)$, and thus there exists $N_2(r,\varepsilon) \in \N$ such that  
	$$
	|\psi_n(r)-s^\infty(r)|<\frac{\varepsilon}{4}, \quad n>N_2(r,\varepsilon).
	$$
	Since $x \not\in D$, the function $s^\infty$ is continuous at $x$. Therefore, there exists $\delta_1(x,\varepsilon) >0$ such that 
	$$
	|s^\infty(r)-s^\infty(x)| < \frac{\varepsilon}{4}
	$$ 
	for all $r \in D$ with $|x-r| < \delta_1(x,\varepsilon)$.
	
	Since $x \notin D$, also the function $F$ is continuous at $x$.  Consequently, the same is true for the function $v_F$  (see Proposition~\ref{Corol_ContEquiv_v_f-f}), and there exists   $\delta_2(x,\varepsilon)> 0$ such that $\LocalModulCont{v_F}{x}{\delta_2(x,\varepsilon)} < \frac{\varepsilon}{4} $.
	Finally, there exists $N_3(x,\varepsilon) \in \N$ such that $|\chi_n|<\frac{1}{4} \delta_2(x,\varepsilon)$ for $n>N_3(x,\varepsilon)$.
	
	Now, choose and fix $r_0(x) \in D$ that satisfies
	$$
	|x - r_0(x)| < \min \left\{ \delta_1(x,\varepsilon), \frac{1}{4} \delta_2(x,\varepsilon) \right\}.
	$$
	By Lemma~\ref{Lemma_LocModuli_c_n<=LocModuli_v_F} we have
	\begin{align*}
	& |\psi_n(x)-\psi_n(r_0(x))| \le \LocalModulCont{\psi_n}{x}{2|x-r_0(x)|} \le \LocalModulCont{v_F}{x}{2|x-r_0(x)|+2|\chi_n|} \\ 
	& \le \LocalModulCont{v_F}{x}{\delta_2(x,\varepsilon)} < \frac{\varepsilon}{4}, \quad n > N_3(x,\varepsilon).
	\end{align*}
	Taking $n > \max\{N_1(x,\varepsilon), N_2(r_0(x),\varepsilon), N_3(x,\varepsilon) \}$ in \eqref{formula2} we thus obtain
	$$
	|\psi^\infty(x)-s^\infty(x)| < \varepsilon.
	$$
	Hence, $\psi^\infty(x)=s^\infty(x)$ for all $x\in[a,b]$. Thus, $s^\infty$ is a pointwise limit of a sequence of chain functions of $F$ and therefore  a metric selection of $F$. 
\end{proof}
\medskip


{ \large \bf Appendix B: Proof of Theorem~\ref{Theo_FourierConv_Class_Jump}} \label{Section_Appendix_B}
\medskip

\noindent \textbf{Theorem~\ref{Theo_FourierConv_Class_Jump}.}
\textit{ Let $B > 0$, $x \in \R$ and $\omega$ be a modulus-bounding  function. Then for all ${f \in \EqvClass{1}{B}{x}{\omega}}$, each $n \in \N$ and each  $\delta\in(0,\pi]$  we have
	$$
	\left|  \mathscr{S}_n f(x) - \frac{1}{2}\big ( f(x+0) + f(x-0) \big ) \right| \le \frac{2B}{\pi n} \left( 1+6\cot \left(\frac{\delta}{2} \right) \right)+ 8C\omega(\delta),
	$$
	where $C$ is the constant from Lemma~\ref{Lemma:Int_D_Estimate}.
}
\begin{proof}
	Let $f$ be an arbitrary function from the class $\EqvClass{1}{B}{x}{\omega}$.
	By \eqref{D-D*} we have
	$$
	\mathscr{S}_n f(x) -  \mathscr{S}^*_n f(x) = \frac{1}{2\pi} \int_{-\pi}^{\pi} \cos{n(x-t)} \, f(t) dt = \frac{1}{2} (\cos{nx} \, a_n(f) + \sin{nx} \, b_n(f)),
	$$
	where $a_n(f)$  and $b_n(f) $ are the cosine and sine Fourier coefficients \eqref{Fourier-a,b-coeff} of the function $f$. By Lemma~\ref{Lemma:BV_Estimate_ab} we have
	\begin{equation} \label{S-S*}
	|\mathscr{S}_n f(x) -  \mathscr{S}^*_n f(x)| \le \frac{1}{2} ( |a_n(f)| + |b_n(f)|)  \le \frac{2B}{\pi n}. 
	\end{equation}
	Next we estimate $\left|  \mathscr{S}^*_n f(x) - \frac{1}{2}(f(x+0) + f(x-0)) \right|$ in two steps.
	
	\noindent \underline{Step 1.}
	Consider the functions
	$$
	\varphi(t) = \frac{1}{2} ( f(x+t) + f(x-t)), \quad \psi(t) = \frac{1}{2} ( f(x+t) - f(x-t)). 
	$$
	The functions $\varphi$ and $\psi$ are $2\pi$-periodic; $\varphi$ is even with $\varphi(0) = f(x)$ and ${\lim_{t \to 0}{\varphi(t)} = \frac{1}{2} ( f(x+0) + f(x-0))}$, $\psi$ is odd with $\psi(0) = 0$.
	Clearly, $\varphi(t) + \psi(t) = f(x+t)$. The function $D^*_n$  is even. Taking these facts into account we get
	\begin{align*}
	& \mathscr{S}^*_n f(x) = \frac{1}{\pi} \int_{-\pi}^{\pi} D^*_n(x-t)  f(t) dt = \frac{1}{\pi} \int_{-\pi}^{\pi} D^*_n(t-x)  f(t) dt  = \frac{1}{\pi} \int_{-\pi}^{\pi} D^*_n(t)  f(x+t) dt \\&= \frac{1}{\pi} \int_{-\pi}^{\pi} D^*_n(t)  \varphi(t) dt + \frac{1}{\pi} \int_{-\pi}^{\pi} D^*_n(t)  \psi(t) dt  = \frac{2}{\pi} \int_{0}^{\pi} D^*_n(t)  \varphi(t) dt.
	\end{align*}
	Now introduce for $t \in [0,\pi]$ the function
	$$
	r(t) = \begin{cases}
	\varphi(t) - \frac{1}{2} ( f(x+0) + f(x-0)), & t > 0 ,\\
	0, & t = 0.
	\end{cases}
	$$
	This function is continuous at $t=0$. In view of~\eqref{D-int} we get
	$$
	\mathscr{S}^*_n f(x) - \frac{1}{2}(f(x+0) + f(x-0)) =  \frac{2}{\pi} \int_{0}^{\pi} D^*_n(t)  r(t) dt.
	$$
	For $t \in [0,\pi]$, we define $r^+(t) = V_0^t (r)$ and $r^-(t) = r^+(t) - r(t)$. 
	We extend $r^+$, $r^-$ to $[-\pi,0)$ such that the resulting functions are even.  
	The functions $r^+$, $r^-$ are both non-negative and non-decreasing in $[0,\pi]$. 
	In terms of these functions we obtain the representation
	\begin{equation}\label{Formule-repres}
	\mathscr{S}^*_n f(x) - \frac{1}{2}(f(x+0) + f(x-0)) =  \frac{2}{\pi} \int_{0}^{\pi} D^*_n(t)  r^+(t) dt - \frac{2}{\pi} \int_{0}^{\pi} D^*_n(t)  r^-(t) dt.
	\end{equation}
	
	\noindent \underline{Step 2.}
	We derive bounds for the two integrals in~\eqref{Formule-repres}. 
	
	First we estimate the variations of the functions $r$, $r^+$, $r^-$.
	Due to the continuity of $r$ at zero, for its local variation  on the interval $[0,\delta]$, $0 < \delta \le \pi$, we have
	\begin{align*}
	V_0^\delta(r) &= \lim_{t \to 0+}{ V_t^\delta(r)}
	\le \frac{1}{2} \left(  \lim_{t \to 0+}{ V_{x+t}^{x + \delta}(f)} + \lim_{t \to 0+}{ V_{x - \delta}^{x -t}(f)} \right) \le \frac{1}{2} \left( \NewRightLocalModul{v_f}{x}{\delta} + \NewLeftLocalModul{v_f}{x}{\delta} \right)
	\le \omega(\delta).
	\end{align*}
	It follows that
	\begin{equation}\label{Formule_r-estim_1}
	\begin{split}
	&V_0^\delta(r^+)  = r^+(\delta) = V_0^\delta(r) \le \omega(\delta),\\&
	r^-(\delta) = V_0^\delta(r^-) \le V_0^\delta(r^+) + V_0^\delta(r) \le 2\omega(\delta).
	\end{split}
	\end{equation}
	For the global variation of these functions on $[0,\pi]$ or $[-\pi,\pi]$, respectively, we obtain the estimates
	\begin{equation}\label{Formule_r-estim_2}
	\begin{split}
	&V_{0}^{\pi}(r) \le \frac{1}{2} \left( V_0^{\pi}(f) + V_{-\pi}^0(f) \right) = \frac{1}{2} V_{-\pi}^{\pi}(f) \le \frac{1}{2}B,
	\\&
	V_{-\pi}^{\pi}(r^+) = 2 V_{0}^{\pi}(r^+) =  2V_0^{\pi}(r)   \le B,
	\\&
	V_{-\pi}^{\pi}(r^-) = 2 V_{0}^{\pi}(r^-) \le 2 \left( V_{0}^{\pi}(r^+) + V_{0}^{\pi}(r) \right) \le 2B.
	\end{split}
	\end{equation}
	Since the analysis of the two integrals in~\eqref{Formule-repres} is similar, we denote both of the integrals by $\int_{0}^{\pi} D^*_n(t)  r^\pm(t) dt$.
	Fix $\delta \in (0,\pi]$ and write
	$$
	\frac{2}{\pi} \int_{0}^{\pi} D^*_n(t)  r^\pm(t) dt = \frac{2}{\pi} \int_{0}^{\delta} D^*_n(t)  r^\pm(t) dt + \frac{2}{\pi} \int_{\delta}^{\pi} D^*_n(t)  r^\pm(t) dt = \mathcal{A}^\pm + \mathcal{B}^\pm.
	$$
	To estimate $\mathcal{A}^\pm$, we use the second mean value theorem,
	$$
	\mathcal{A}^\pm = \frac{2}{\pi} \int_{0}^{\delta} D^*_n(t)  r^\pm(t) dt
	= r^\pm(\delta-0) \frac{2}{\pi} \int_{\delta^\pm}^{\delta} D^*_n(t)   dt 
	$$
	with some $0 < \delta^\pm < \delta$, and thus by \eqref{Int_D_Estimate} and~\eqref{Formule_r-estim_1}
	$$
	|\mathcal{A}^\pm| \le | r^\pm(\delta)| \cdot \left|   \frac{2}{\pi} \int_{\delta^\pm}^{\delta} D^*_n(t)   dt  \right| <C| r^\pm(\delta)|\le 4 C\omega(\delta).
	$$
	For the second term $\mathcal{B}^\pm$ we have in view of~\eqref{Formule_ModifKernel} 
	$$
	\mathcal{B}^\pm =  \frac{2}{\pi} \int_{\delta}^{\pi}  \frac{1}{2} \sin{nt} \cot{ \left( \frac{1}{2}t \right)} r^\pm(t) dt
	= \frac{1}{\pi} \int_{-\pi}^{\pi}  \sin{nt} \, u_{\delta}(t)r^\pm(t) dt,
	$$
	where
	$$
	u_{\delta}(t) =  
	\begin{cases}
	\frac{1}{2} \cot{ \left( \frac{1}{2}t \right)}, & \delta \le |t| \le \pi ,\\
	0, & \mbox{otherwise}.
	\end{cases}
	$$
	In other words, $\mathcal{B}^\pm$ are the sine Fourier coefficients \eqref{Fourier-a,b-coeff} of the functions $u_{\delta}(t) r^\pm(t)$. The function $\frac{1}{2} \cot{ \left( \frac{1}{2}t \right)}$ is odd, decreasing in the interval $[\delta, \pi]$ and takes the value zero at $\pi$. Consequently, ${ \left | \frac{1}{2} \cot{ \left( \frac{1}{2} t \right)}\right | \le \frac{1}{2} \cot{ \left( \frac{1}{2} \delta \right)} }$ for  $\delta \le |t| \le \pi$ so that
	$$
	\|u_{\delta}\|_\infty \le  \frac{1}{2} \cot{ \left( \frac{1}{2} \delta \right)}
		\quad \text{and} \quad
	V_{-\pi}^\pi \left( u_{\delta} \right) \le 2\cot{ \left( \frac{1}{2} \delta \right)}. 
	$$
	On the other hand, for $t \in [0,\pi]$ we have by~\eqref{Formule_r-estim_2} $0 \le r^+(t) \le r^+(\pi) = V_0^\pi (r) \le \frac{1}{2}B$ and $0 \le r^-(t) \le r^-(\pi) = V_0^\pi (r^-) \le B$, so that 
	$$
	\| r^\pm \|_\infty \le B.  
	$$
	Using the inequality $V_a^b(gh) \le \|g\|_{\infty} V_a^b(h) + \|h\|_\infty V_a^b(g) $, we obtain
	$$
	V_{-\pi}^\pi (u_{\delta} \cdot r^\pm) \le  3B \cot{ \left( \frac{1}{2} \delta \right)}.
	$$
	Applying Lemma~\ref{Lemma:BV_Estimate_ab} again we estimate
	$$
	|\mathcal{B}^\pm| \le 6B \cot{ \left( \frac{1}{2} \delta \right)}  \frac{1}{\pi n}.
	$$
	Finally, taking into account~\eqref{S-S*} we obtain
	\begin{align*}
	& \left|  \mathscr{S}_n f(x) -  \frac{1}{2}(f(x+0) + f(x-0))   \right| 
	\le \left|  \mathscr{S}_n f(x) - \mathscr{S}^*_nf(x)   \right|  +  |\mathcal{A}^+| + |\mathcal{B}^+| +  |\mathcal{A}^-| + |\mathcal{B}^-|
	\\& \le \frac{2B}{\pi n} + 8 C\omega(\delta) + 12B\cot{\left(\frac{1}{2}\delta \right )}\frac{1}{\pi n}=\frac{2B}{\pi n}(1+6\cot(\delta/2))+8C\omega(\delta)
	\end{align*}
	which is the desired estimate.
\end{proof}


\end{document}